\documentclass{amsart}

\usepackage{latexsym, amsthm}
\usepackage{rotating, amscd}

\usepackage{enumerate, amssymb}
\usepackage{amsmath}
\usepackage{graphics, color}
\definecolor{darkblue}{rgb}{0,0,0.4}
\usepackage[colorlinks=true, citecolor=darkblue, filecolor=darkblue,
linkcolor=darkblue, urlcolor=darkblue]{hyperref}
\input xy
\xyoption{all}
\DeclareMathOperator*{\colim}{colim}

\newtheorem{theorem}{Theorem}[section] 
\newtheorem{lemma}[theorem]{Lemma}     
\newtheorem{corollary}[theorem]{Corollary}
\newtheorem{proposition}[theorem]{Proposition}
\newtheorem{remark}[theorem]{Remark}
\newtheorem{definition}[theorem]{Definition}
\newtheorem{question}[theorem]{Question}
\newtheorem{example}[theorem]{Example}

\newtheorem{alphatheorem}{Theorem}

\newcommand{\fX}{\mathfrak{X}}
\newcommand{\fQ}{\mathfrak{Q}}

\newcommand{\bbC}{\mathbb{C}}
\newcommand{\bbF}{\mathbb{F}}
\newcommand{\C}{\mathbb{C}}

\newcommand{\R}{\mathbb{R}}
\newcommand{\bbR}{\mathbb{R}}

\newcommand{\Z}{\mathbb{Z}}
\newcommand{\bbZ}{\mathbb{Z}}

\newcommand{\ignore}[1]{}

\newcommand{\leqs}{\leqslant}
\newcommand{\geqs}{\geqslant}
\newcommand{\heq}{\simeq}

\newcommand{\maps}{\longrightarrow}

\newcommand{\injects}{\hookrightarrow}

\newcommand{\isom}{\cong}
\newcommand{\cross}{\times}

\newcommand{\wt}[1]{\widetilde{#1}}

\newcommand{\homeo}{\isom}

\newcommand{\Rdef}{R^{\mathrm{def}}}

\newcommand{\Hom}{\mathrm{Hom}}

\newcommand{\GL}{\mathrm{GL}}

\newcommand{\Un}{\mathrm{U}}
\newcommand{\SP}{\mathrm{SP}}
\newcommand{\SUn}{\mathrm{SU}}
\newcommand{\GLn}{\mathrm{GL}}
\newcommand{\SLn}{\mathrm{SL}}
\newcommand{\K}{K^{\mathrm{def}}}

\newcommand{\Map}{\mathrm{Map}}

\newcommand{\Id}{\mathrm{Id}}

\newcommand{\xmaps}{\xrightarrow}
\newcommand{\srm}[1]{\stackrel{#1}{\maps}}
\newcommand{\srt}[1]{\stackrel{#1}{\to}}

\newcommand{\goesto}{\mapsto}
\newcommand{\nd}{\noindent}

\newcommand{\ol}[1]{\overline{#1}}
\def\co{\colon\thinspace}

\newcommand{\Img}{\mathrm{Im}}
 
\newcommand{\e}{\emph}

\newcommand{\quot}{/\!\!/}

\newcommand{\Si}{\Sigma}

\newcommand{\bZ}{\mathbb{Z}}
\newcommand{\bR}{\mathbb{R}}
\newcommand{\bC}{\mathbb{C}}

\newcommand{\fh}{\mathfrak{h}}

\newcommand{\tG}{\tilde{G}}
\newcommand{\tK}{\tilde{K}}

\newcommand{\ta}{\tilde{a}}
\newcommand{\tb}{\tilde{b}}

\newcommand{\rahl}{\rightarrow}
\newcommand{\Ker}{\mathrm{Ker}}

\newcommand{\holG}{\Hom(\pi_1(\Si),G)}
\newcommand{\holGss}{\Hom(\pi_1(\Si),D)}
\newcommand{\holK}{\Hom(\pi_1(\Si),K)}

\newcommand{\ab}{a_1,b_1,\ldots, a_\ell,b_\ell}

\newcommand{\tabab}{\prod_{i=1}^\ell[\ta_i,\tb_i]}

\title[Covering spaces of character varieties]{Covering spaces of character varieties}

\date{\today}

\author[S. Lawton]{Sean Lawton}

\address{Department of Mathematical Sciences, George Mason University,
4400 University Drive,
Fairfax, Virginia  22030, USA}

\email{slawton3@gmu.edu}

\author[D. Ramras]{Daniel Ramras\\ \\
{\tiny With an Appendix by Nan-Kuo Ho and Chiu-Chu Melissa Liu}}

\address{Daniel Ramras, Department of Mathematical Sciences, Indiana University-Purdue University Indianapolis, 402 N. Blackford, LD 270, 
Indianapolis, IN 46202, USA}

\email{dramras@math.iupui.edu}

\address{Nan-Kuo Ho,
Department of Mathematics, National Tsing Hua University and
National Center for Theoretical Sciences, Hsinchu 300, Taiwan}

\email{nankuo@math.nthu.edu.tw}

\address{Chiu-Chu Melissa Liu,
Department of Mathematics, Columbia University,
2990 Broadway, New York, NY 10027, USA}

\email{ccliu@math.columbia.edu}


 
\thanks{The first author was partially supported by the Simons Foundation (\#245642) and the NSF-DMS (\#1309376). The second author was partially supported by the Simons Foundation (\# 279007).}

\begin{document}

\begin{abstract}
Let $\Gamma$ be a finitely generated discrete group.
Given a covering map $H\to G$ of Lie groups with $G$ either compact or complex reductive, there is an induced covering map  
$\Hom(\Gamma, H)\to \Hom(\Gamma, G)$.  We show that  when $\pi_1 (G)$ is torsion-free
and $\Gamma$ is free, free Abelian, or the fundamental group of a closed Riemann surface $M^g$, 
this map induces a covering map between the corresponding moduli spaces of representations.  
We give conditions under which this map is actually the \e{universal} covering, 
leading to new information regarding  fundamental groups of these moduli spaces.  As an application,
we show that for $g\geqs 1$, the stable moduli space $\Hom(\pi_1 M^g, \SUn)/\SUn$ is homotopy equivalent to $\bbC \textrm{P}^\infty$.

In the Appendix, Ho and Liu show that $\pi_0 \Hom( \pi_1 M^g, G)$ and $\pi_1 [G,G]$ are in bijective correspondence for all complex connected reductive Lie groups $G$.
\end{abstract}

\maketitle

\section{Introduction}

Let $\Gamma$ be a finitely generated discrete group, and let $G$ be a Lie group.  In this article, our main object of interest is the fundamental group of the moduli space of representations $$\fX_\Gamma(G)=\Hom(\Gamma, G)\quot G.$$  When $G$ is compact, this quotient can be interpreted as the point-set quotient space for the conjugation action of $G$.  For non-compact $G$, we define $\fX_\Gamma (G)$ to be the \e{polystable quotient}, which is the subspace of the point-set quotient consisting of closed points.  When $G$ is a complex reductive algebraic group, this subspace is homeomorphic to the GIT quotient, when the latter is equipped with the Euclidean topology (see \cite{FlLaAbelian} for a proof, and \cite{Do} for generalities on GIT). Whenever $G$ is the real-locus of a complex reductive algebraic group defined over $\R$ (a situation including both compact Lie groups and complex reductive algebraic groups)
the polystable quotient is a real semi-algebraic set \cite{RiSl}, and thus deformation retracts to a finite simplicial 
complex~\cite[Corollary 9.3.7 and Remark 9.3.8]{BCR}; hence its fundamental group is finitely presented.  We will also consider the point-set quotient space for the conjugation action of $G$, denoted by $\fQ_{\Gamma}(G)$.

Here is our main theorem (Theorem~\ref{prop-discont}).  Recall that a covering space is normal (or regular) if its group of deck transformations acts transitively on every fiber.
\begin{alphatheorem}
Let $G$ be either a connected reductive algebraic group over $\C$, or a compact connected Lie group, and assume that $\pi_1 (G)$ is torsion-free.   Let $p\co H\to G$ be a covering homomorphism.
Then the induced map $p_* \co \fQ_\Gamma (H) \to \fQ_\Gamma (G)$ is a normal covering map {\rm (}onto its image{\rm )} with structure group
$\Hom(\Gamma, \ker(p))$, and restricts to a normal covering map $p_* \co \fX_\Gamma (H) \to \fX_\Gamma (G)$ {\rm (}onto its image{\rm )}, with the same structure group.
\end{alphatheorem} 

This result is proven in Section~\ref{ModuliSection}, where it is also shown that the images of these covering maps are unions of path components inside their codomains.

We  discuss applications of this theorem to universal covers and fundamental groups of moduli spaces in Section~\ref{univ-cover-sec}.   This section focuses on the class of {\it exponent-canceling} groups.  A finitely generated group $\Gamma$ is called  exponent-canceling if the Abelianization of each of its relations is trivial; for instance free groups, free Abelian groups, right-angled Artin groups, and fundamental groups of closed Riemann surfaces are exponent-canceling.  Under this condition,  the above covering maps are in fact \e{surjective} (Proposition~\ref{surj}).

In particular, we compute these fundamental groups for free groups, free Abelian groups, and surface groups, subject to certain conditions on the Lie group $G$.  For instance, we prove:

\begin{alphatheorem}$\label{Bthm}$  Let $\Sigma$ be a compact, orientable surface {\rm (}with or without boundary{\rm )} and let $G$ be a connected complex reductive algebraic group with $\pi_1 (G)$ torsion-free.
Then 
$$\pi_1 (\fX_{\pi_1 \Sigma} (G)) \isom \pi_1 (G)^{b_1 (\Sigma)},$$
where $b_1 (\Sigma)$ is the first Betti number.

Furthermore, if $T = (S^1)^r$ and the derived subgroup $[G,G]$ is a product of groups of type $A_n$ or $C_n$, then 
$$\pi_1 (\fX_{\pi_1 T} (G)) \isom \pi_1 (G)^{b_1 (T)}.$$
\end{alphatheorem}

Connected components of character varieties have been studied extensively; in the Appendix by Ho and Liu, work of Jun Li \cite{Li-surface-groups} is extended to complete the picture for complex representations of surface groups.   We note this result has recently been generalized in \cite{GaOl} in the context of Higgs bundles (also see \cite{DoPa} for related work).

To our knowledge, the present article is the first systematic study of \e{fundamental groups} of character varieties.  In joint work with Biswas~\cite{BiLaRa}, we have extended Theorem~\ref{Bthm} to the case in which $\pi_1 (G)$ has torsion (see also Biswas--Lawton~\cite{BiLa}). 

In the case of closed surfaces, Theorem~\ref{Bthm} can be viewed as an extension of results of Bradlow--Garc\'{i}a-Prada--Gothen~\cite{BGG}, who studied
homotopy groups of character varieties for \e{central extensions} of surface groups using the Morse theory of Higgs bundles.  Let $\Gamma^g$ denote the fundamental group of a closed orientable surface $\Sigma^g$ of genus $g$, with presentation
$$\Gamma^g = \langle a_1, \ldots, a_g, b_1, \ldots b_g \, : \, \prod_i [a_i, b_i] = 1\rangle.$$  
Let $\widehat{\Gamma^g}$ denote the central extension of $\Gamma^g$ by $\bbZ$, in which the central generator is set equal to $\prod_i [a_i, b_i]$.  The path components of moduli space $\fX_{\widehat{\Gamma^g}} (\GLn (n))$ are determined by the image $d$ of the central generator in $Z (\GLn (n))\cap [\GLn (n), \GLn (n)] = Z (\SLn (n)) = \bbZ/n\bbZ$, which is called the \e{degree} of the representation.  When $n$ and $d$ are coprime, \cite[Theorem 4.4 (1)]{BGG} shows that the corresponding path component of $\fX_{\widehat{\Gamma^g}} (\GLn (n))$ has fundamental group  $H_1 (\Sigma^g; \bbZ) = \pi_1 (\GLn (n))^{b_1 (\Sigma^g)}$, and our result extends this to the degree zero component (note that degree zero representations of $\widehat{\Gamma^g}$ are simply representations of $\Gamma^g$).

In the complex case, our proof of Theorem~\ref{Bthm} uses a path-lifting result for GIT quotients (Section~\ref{lifting-sec}).

 In Section~\ref{stable-sec}, we apply these results to study  \e{stable} moduli spaces of surface group representations.  In particular, we show that there is a homotopy equivalence
 $$\fX_{\pi_1 M^g} (\SUn) := \Hom(\pi_1 M^g, \SUn)/\SUn \heq \bbC \textrm{P}^\infty,$$ 
 where $M^g$ is a closed Riemann surface of genus $g\geqs 1$, and we identify a generator for 
 $\pi_2 (\fX_{\pi_1 M^g} (\SUn)) \isom \bbZ$.  Additionally, we prove $\pi_1 (\fX_{\pi_1 M^g} (\GLn)) \isom \bbZ^{2g}$.

\section*{Acknowledgments}
The authors thank Chiu-Chu Melissa Liu and Nan-Kuo Ho for helpful conversations, and for sharing the results in the Appendix with us.  Additionally, the authors thank Jos\'{e} Manuel Gomez, and the anonymous referees for helpful comments.

\section{Universal covers and fundamental groups of Lie groups}

In this section, we record some basic facts regarding the structure of the universal covering homomorphism for compact Lie groups and for complex reductive algebraic groups (which we call \e{reductive $\bbC$--groups} for short), as well as some facts regarding torsion in the fundamental group.

An affine algebraic group is called \emph{reductive} if it does not contain any non-trivial closed connected unipotent normal subgroup (see \cite{Bo,hum-book} for generalities on algebraic groups). Since we do not consider Abelian varieties in this paper, we will abbreviate the term \textit{affine algebraic group} to simply \textit{algebraic group.}  Over $\C$, each reductive algebraic group $G$ arises as the Zariski closure of a compact Lie group $K$.  In particular, as discussed in \cite{Sch1},
we may assume $K\subset\mathrm{O}(n,\R)$ is a real affine variety by the Peter-Weyl theorem, and thus $G\subset\mathrm{O}(n,\C)$ is the complex points of the real variety $K$.  Note that some of our references concern {\it linearly reductive groups}, that is, algebraic groups such that all linear representations are completely reducible.  However, over $\C$ linearly reductive groups and reductive groups are the same (see the discussion of Haboush's Theorem in \cite{Do}, or \cite[p. 189]{Procesi-book}).

By a simple Lie group, we will mean a connected Lie group with no non-trivial, connected normal subgroups.  With this definition, simple Lie groups are precisely those whose Lie algebra is simple, meaning that it has no non-trivial ideals.

\begin{lemma}$\label{central}$ Let $f\co K\to H$ be a covering homomorphism of Lie groups.  Then $\ker (f)$ is central in $K$.
\end{lemma}  Since $\ker (f)$ is discrete and normal, this follows from~\cite[Theorem 6.13]{Hofmann-Morris}.

\begin{proposition}$\label{univ-cover}$ Let $G$ be a compact, connected Lie group {\rm (}respectively, a connected reductive $\bbC$--group{\rm )}, and let $\bbF = \bbR$ {\rm (}respectively, let $\bbF = \bbC$ {\rm )}.
Then for some $k\geqs 0$ and some compact {\rm (}respectively, complex reductive algebraic{\rm )}, simple, simply connected Lie groups $G_1, \ldots, G_l$, there is a universal covering homomorphism
$$\wt{G} = \bbF^k \cross G_1 \cross \cdots \cross G_l \srm{q} G.$$
The subgroup $G_1 \cross \cdots \cross G_l  \leqs \wt{G}$ is precisely the commutator subgroup $D\wt{G}$, and the restriction of $q$ to $D\wt{G}$ is a universal covering homomorphism $q\co D\wt{G} \to DG$.

Now assume $\pi_1 (G)$ is torsion-free, and let 
$\pi\co \wt{G} \to \bbF^k$ be the projection map.
Then the restriction of $\pi$ to $\ker(q)$ is injective, and $\pi (\ker(q))$ is a discrete subgroup of $\bbF^k$.
It follows that $q\co D\wt{G}\to DG$ is an isomorphism, and hence $DG$ is simply connected.  In particular, the universal covering map $q$ has the form
$$q\co \bbF^k \cross DG \maps G,$$
and the restriction of $q$ to $DG$ is simply the inclusion $DG\injects G$.
\end{proposition}

\begin{proof}
In the compact case~\cite[Theorem 6.19]{Hofmann-Morris} states that there exist
connected, compact, simple Lie groups $H_1,\dots, H_l$ and a short exact sequence
$$1\maps Z \maps (S^1)^k \cross H_1 \cross \cdots \cross H_l \maps G \maps 1$$
with $Z$ finite and central.  Let $G_i$ denote the universal cover of $H_i$.  We claim that $G_i$ is  compact, connected, and simple.  Simplicity follows from the fact that $G_i$ and $H_i$ have the same Lie algebra; connectivity of $G_i$ follows immediately from the fact that $H_i$ is connected; and $G_i$ is compact by~\cite[Theorem 0.6.10]{Bredon-transf}.

The map $q$ is now the composite of  two covering maps
$$\wt{G} = \bbR^k \cross G_1 \cross \cdots \cross G_l \srm{q_1} (S^1)^k \cross H_1 \cross \cdots \cross H_l \srm{q_2} G.$$
The map $q_1$ is the product of universal covering maps $G_i \to H_i$ and $\exp\co \bbR^k\to (S^1)^k$, and
the map $q_2$  is the quotient map for $Z$.  Since each $G_i$ is simply connected, it follows that $q=q_2 \circ q_1$ is a universal covering homomorphism.

To see that $D\wt{G} = G_1 \cross \cdots \cross G_l$, first note that
since each $G_i$ is simple, the product $G_1 \cross \cdots \cross G_l$ is semisimple, and hence equal to its commutator subgroup by Goto's Commutator Theorem (see, for instance,~\cite[Theorem 6.56]{Hofmann-Morris}).
Now
 \begin{equation} \label{D} D\wt{G}  = D(\bbF^k \cross G_1 \cross \cdots \cross G_l) = D\bbF^k\cross D (G_1 \cross \cdots \cross G_l)  = G_1 \cross \cdots \cross G_l.\end{equation}

Next, we check that the restriction of $q$ to $D\wt{G}$ gives the universal covering of $DG$.  By (\ref{D}), we know $D\wt{G}$ is simply connected,  so we just need to check that $q (D\wt{G}) = DG$.
 The containment $q(D\wt{G}) \subset DG$ is immediate.  To see that $DG\subset q(D\wt{G})$, note that by surjectivity of $q$, each element of $G$ has the form $q(z) q(c)$, where $z\in \bbF^k$ is central in $\wt{G}$ and $c\in D\wt{G}$ is a commutator.  Surjectivity of $q$ also shows that $q(z)$ is central in $G$.  Given another element $q(z')q(c') \in G$ (with $z'$ central in $\wt{G}$ and $c'\in D\wt{G}$) we have
$$[q(z)q(c), q(z')q(c')] = [q(c), q(c')] = q([c,c']) \in q(D\wt{G}),$$
as desired.

Now say $\pi_1 (G) \isom \ker (q)$ is torsion-free.
If $x\in \ker(q)$ and  $\pi(x) = 0 \in \bbR^k$ (where $\pi\co \wt{G}\to \bbR^k$ is the projection), then  
by Lemma~\ref{central}, $x$ lies in the center
$$Z ( G_1 \cross \cdots \cross G_l) = Z(G_1) \cross \cdots \cross Z(G_l).$$
Since $G_i$ is simple, the identity component of its center is trivial, meaning that the center is a discrete group.  Since each $G_i$ is also compact, we conclude that $Z(G_i)$ is finite for each $i$.  Hence $x$ is a torsion element in $\ker(q)$, so $x=1$.  This shows that the restriction of $\pi$ to $\ker(q)$ is an injective map, as claimed.
Moreover, if $x\in \ker(q)$ then $q_1 (x) \in Z = \ker(q_2)$, so projecting $q_1 (x)$ to $(S^1)^k$ gives an element whose coordinates are $n$th roots of unity, where $n=|Z|$.  
Letting $\mu_n \leqs S^1$ denote the subgroup of $n$th roots of unity, we have $\pi(x) \in \exp^{-1} ((\mu_n)^k)$, which is a discrete subgroup of $\bbR^k$.  Hence $\pi(\ker(q)) \leqs \exp^{-1} ((\mu_n)^k) \leqs \bbR^k$ is discrete as well.

Since $\ker (q)$ intersects $D\wt{G} = G_1 \cross \cdots G_l$ trivially, $q$ restricts to an isomorphism  $D\wt{G} \to DG$.  This completes the proof in the compact case.

In the complex case, there exists a short exact sequence of Lie groups
$$1\maps A \maps (\bbC^*)^k \cross G_1 \cross \cdots \cross G_l \srm{p} G\maps 1$$
where each $G_i$ is a simply connected, simple Lie group, and the kernel $A$ is a finite central subgroup (see for example~\cite[p. 376]{Procesi-book}). 

The quotient map $p$ is a covering map (since $A$ is finite) and now we obtain a universal covering 
homomorphism $q$ for $G$ as the composite
$$\bbC^k \cross  G_1 \cross \cdots \cross G_l \xmaps{\textrm{exp} \cross \Id}  (\bbC^*)^k \cross G_1 \cross \cdots \cross G_l \srm{p} G.$$ 
If $\pi_1 (G) \isom \ker (q)$ is torsion-free, and $x\in \ker(q)$ lies in $G_1 \cross \cdots \cross G_l$, then in fact $x\in\ker (p) = A$.  Since $A$ is finite, this implies $x=1$.  The remainder of the proof in the complex case is the same as in the compact case; note that it is a theorem of Pasiencier--Wang~\cite{Pasiencier-Wang} that every complex semisimple Lie group is equal to its commutator subgroup.  
\end{proof}

In this article, we will mainly be interested in Lie groups $G$ with $\pi_1 (G)$ torsion-free.  In Proposition~\ref{univ-cover} we showed that, in our cases of interest, this condition implies that $DG$ is simply connected.  In fact, the converse is also true.

\begin{proposition}$\label{torsion-free}$
Let $G$ be either a compact, connected Lie group or a connected reductive $\bbC$--group.  Then $\pi_1 (G)$ is torsion-free if and only if the commutator subgroup $DG$ is simply connected.
\end{proposition}
\begin{proof} We only need to show that when $DG$ is simply connected, $\pi_1 (G)$ is torsion-free, so 
say $DG$ is simply connected.    By Proposition~\ref{univ-cover}, the universal covering map $q\co \wt{G} = \bbF^k \cross D\wt{G} \to G$ restricts to a universal covering $D\wt{G} \to DG$.  Since $DG$ is simply connected, this covering must in fact be an isomorphism, so $q$ is injective on $D\wt{G}$.  Hence $\ker (q)$ injects into $\bbF^k$ under the projection $\wt{G}\to \bbF^k$, and since $\bbF^k$ is torsion-free, so is $\ker (q) \isom \pi_1 (G)$.  
\end{proof}

\begin{remark}$\label{torsion-rmk}$
If $G$ is either a compact Lie group or a complex reductive $\bbC$--group, then the inclusion $i\co DG\injects G$ actually induces an isomorphism between $\pi_1 (DG)$ and the torsion subgroup of $\pi_1 (G)$, giving another proof of Proposition~\ref{torsion-free}.  To prove this stronger statement, note that there is a natural exact sequence of groups $1\to DG\to G\srt{p} G/DG\cong Z\to 1$, where $Z$ is a complex or compact torus.  This then gives a long exact sequence in homotopy:
$$0=\pi_2(G/DG)\to\pi_1(DG)\srm{i_*} \pi_1(G)\srm{p_*} \pi_1(G/DG)=\Z^k, $$
which shows that $i_*$ is injective.
Since $\bbZ^k$ is torsion-free, the torsion subgroup of $\pi_1 (G)$ is contained in $i_* (\pi_1 (DG)) \isom \ker (p_*)$.   By Weyl's Theorem, the fundamental group of a semisimple compact Lie group is finite, and the same holds for complex semisimple reductive $\bbC$--groups, since each such group deformation retracts to its maximal compact subgroup {\rm (}which is again semisimple{\rm )}.  Hence $i_* (\pi_1 (DG))$ is contained in the torsion subgroup of $\pi_1 (G)$, completing the argument.
\end{remark}

\section{Moduli spaces of representations}\label{ModuliSection}

Let $\Gamma$ be a finitely generated discrete group and let $G$ be a Lie group.  Consider the \emph{representation space} $\Hom(\Gamma,G)$, the space of homomorphisms from $\Gamma$ to $G$.  We give this space the compact-open topology (i.e. the subspace topology inherited from the mapping space $\Map (\Gamma, G)$, where $\Gamma$ has the discrete topology).
For example, when $\Gamma=F_{r}$ is a free group on $r$ generators, the evaluation of a representation on a set of free generators provides a homeomorphism with the Cartesian product $\Hom(F_{r},G)\cong G^{r}$.

In general, by choosing generators $\gamma_{1},\cdots,\gamma_{r}$ for $\Gamma$ (for some $r$), we have a natural epimorphism $F_{r}\twoheadrightarrow\Gamma$.
This allows one to embed $\Hom(\Gamma,G)\subset\Hom(F_{r},G)\cong G^{r}$, and the Euclidean topology induced from
the manifold $G^{r}$ agrees with the compact-open topology on $\Hom(\Gamma,G)$.

The Lie group $G$ acts on $\Hom(\Gamma,G)$ by conjugation of representations; that is, $g\cdot\rho=g\rho g^{-1}$ for $g\in G$ and $\rho\in\Hom(\Gamma,G)$.  
We define $\fQ_\Gamma (G)$ to be the resulting quotient space.
Let $\Hom(\Gamma,G)^{ps}$ denote the subset of $\Hom(\Gamma,G)$ consisting of representations whose $G$-orbit is closed. Let $\fX_{\Gamma}(G) \subset \fQ_\Gamma (G)$ denote the subspace of closed orbits; that is, $\fX_{\Gamma}(G) =\Hom(\Gamma,G)^{ps}/G$.  Although it is not necessarily an algebraic set, this {\it polystable quotient} is referred to as the \emph{$G$-character variety of $\Gamma$}.  Note that when $G$ is compact, $\Hom(\Gamma,G)^{ps} = \Hom(\Gamma,G)$.

Given a homomorphism of Lie groups $p\co H\to G$, composition with $p$ defines a  map $p_* \co \Hom(\Gamma, H)\to \Hom(\Gamma, G)$.  Since $p(h\rho h^{-1}) = p(h) p_* (\rho) p(h)^{-1}$, we obtain an induced map $\fQ_\Gamma (H) \to \fQ_\Gamma (G)$ which we denote by $\ol{p}_*$, or simply $p_*$ if no confusion is likely.

Two main classes of examples are important for us, and for each the quotient topology on $\fX_{\Gamma}(G)$ will be Hausdorff (in fact, triangulable). When $G$ is a compact Lie group,  $\fX_{\Gamma}(G)= \fQ_\Gamma (G)$ is the usual orbit space (which is semi-algebraic and compact).   When $G$ is a reductive $\C$-group, each orbit closure has a unique closed sub-orbit. So, the quotient $\fX_{\Gamma}(G)$ is set-theoretically identified with the GIT quotient (see \cite{Do}):
\[\Hom(\Gamma,G)\quot G:=\mathrm{Spec}_{max}(\C[\Hom(\Gamma,G)]^{G}).\]  Since affine GIT quotients are affine varieties, they are subsets of Euclidean space and so inherit a Euclidean topology (we will not consider the Zariski topology unless explicitly mentioned).  In these terms, we have the following well-known correspondence between the GIT quotient and the polystable quotient (this fact is implicit in the work of Luna~\cite{Luna}, and an explicit proof appears in \cite[Theorem 2.1]{FlLaAbelian}).

\begin{proposition}
\label{polystablelemma} Let $G$ be a complex reductive algebraic group, and $\Gamma$ be a finitely generated group. Then $\Hom(\Gamma,G)^{ps}/G$ is homeomorphic to $\Hom(\Gamma,G)\quot G$. 
\end{proposition}

Thus in all cases, the spaces $\fX_{\Gamma}(G)$ will be semi-algebraic sets and thus (Hausdorff) simplicial complexes in the natural Euclidean topologies we consider.

\begin{remark}
We note that when $G$ is the set of real points of a reductive $\C$-group, then by \cite{RiSl}, the polystable quotient is a semi-algebraic space.
\end{remark}

We begin the discussion of covering spaces with a result of Goldman~\cite[Lemma 2.2]{Goldman-components}.

\begin{lemma}[Goldman] $\label{Goldman}$
Let $p \co H \to G$ be a covering homomorphism between Lie groups and let $\Gamma$ be a finitely generated group.
Then the image of the induced map
$$\Hom(\Gamma, H) \srm{p_*} \Hom(\Gamma, G)$$
is a union of path components in $ \Hom(\Gamma, G)$.  If $\Hom(\Gamma, G)$ is locally path connected, then 
when viewed as a map to its image, $p_*$ 
is a normal covering map whose group of deck transformations is $\Hom(\Gamma, \ker (p))$.
\end{lemma}

\begin{remark}$\label{Goldman-rmk}$ Lemma~\ref{triang} shows that $\Hom(\Gamma, G)$ is locally path connected in either of the following situations{\rm: 1)} $\Gamma$ is finitely presented {\rm (}and $G$ is an arbitrary Lie group{\rm )}, or \rm{2)} $G$ covers an algebraic group {\rm (}and $\Gamma$ is an arbitrary finitely generated group{\rm )}.

Goldman's article focuses on finitely presented groups, but in fact his argument just requires $\Hom(\Gamma, G)$ to be locally path connected.  

To find an evenly covered neighborhood of a representation $\rho\in \Hom(\Gamma, G) \subset G^r$, first
take a neighborhood  $U\subset G^r$ with $\rho \in U$ such that $U$ is  evenly covered by $p^r$.  By local path connectedness, there exists a path connected neighborhood $V\subset \Hom(\Gamma, G)$ with $\rho\in V \subset U$.  Goldman shows that  $p_*$ has the unique path-lifting property\footnote{This alone is not sufficient to show that a map is a covering map {\rm (}see~\cite[Section 1.3, Exercise 6]{Hatcher}{\rm )}.}, and it follows that 
$\Hom(\Gamma, H)$ is a union of path components inside $(p^r)^{-1} (\Hom(\Gamma, G))$.  Hence $V$ is evenly covered by $p_*$.

\end{remark}

\begin{lemma}$\label{triang}$ Let $G$ be a Lie group and let $\Gamma$ be a discrete group.
If $\Gamma$ is finitely presented, then $\Hom(\Gamma, G)$ is locally path connected.
If $G$ admits a covering homomorphism onto a real algebraic group and $\Gamma$ is finitely generated, then $\Hom(\Gamma, G)$ is triangulable, and hence locally path connected.  
\end{lemma}
\begin{proof}
When $\Gamma$ is finitely presented as $\langle \gamma_1,...,\gamma_r\ |\ R_1,...,R_k\rangle$, $\Hom(\Gamma, G)$ is the locus of solutions to $R_i(g_1,...g_r)=\Id$ in $G^r$.  However, $G^r$ is an analytic manifold (see \cite{Knapp}), and thus $\Hom(\Gamma, G)$ is an analytic variety.  As shown in \cite{Sato}, analytic varieties are locally triangulable, and hence locally path connected.

For the second statement,  let $f\co G\to \ol{G}$ be a covering map, with $\ol{G}$ algebraic.  Then $\Hom(\Gamma, \ol{G})$ is algebraic, and hence triangulable.  Now Lemma~\ref{Goldman} tells us that $f_* \co \Hom(\Gamma, G) \to \Hom(\Gamma, \ol{G})$ is a covering map onto its image, and this image is a union of path components in $\Hom(\Gamma, \ol{G})$.  So the image of $f_*$ is triangulable, and any triangulation can be lifted to the covering space $\Hom(\Gamma, G)$.
\end{proof}

\begin{theorem}\label{prop-discont}
Let $G$ be either a connected reductive $\C$-group or a compact connected Lie group, and assume that $\pi_1 (G)$ is torsion-free.  Let $\Gamma$ be a finitely generated discrete group. Then for any covering homomorphism $p\co H\to G$, the induced maps
$$\ol{p}_* \co \fQ_\Gamma (H) \to \ol{p}_* ( \fQ_\Gamma (H)) \subset \fQ_\Gamma (G)$$ 
and 
$$\ol{p}_* \co \fX_\Gamma (H) \to \ol{p}_* (\fX_\Gamma (H)) \subset  \fX_\Gamma (G)$$
are normal covering maps with structure group $\Hom(\Gamma, \ker(p))$, and we have
\begin{equation} \label{img} \ol{p}_* ( \fX_\Gamma (H) ) = \ol{p}_* (\fQ_\Gamma (H)) \cap  \fX_\Gamma (G). \end{equation}
Moreover, $\ol{p}_*( \fQ_\Gamma (H))$ is a union of path components of $ \fQ_\Gamma (G)$, and $\ol{p}_*( \fX_\Gamma (G))$ is a union of path components of $ \fX_\Gamma (G)$.

\end{theorem}

In general, the base space or the total space for these coverings could be disconnected.  
We note that recent work of Sikora~\cite[Proposition  1]{Sikora-SOn} discusses the relationship between
$\fX_\Gamma (H)$ and $\fX_\Gamma (G)$ for finite covers $H\to G$ and arbitrary $\Gamma$, and Proposition 2 of the same work discusses the case where $\Gamma$ is free.

\begin{proof}  The proof will be broken down into several lemmas.

The group $K := \Hom(\Gamma, \ker (p))$ acts continuously on $\Hom(\Gamma, H)$ by multiplication: for $\phi\in\Hom (\Gamma, \ker (p))$ and $\rho\in \Hom(\Gamma, H)$, we set
$$\phi\cdot \rho (\gamma) = \phi (\gamma) \rho (\gamma).$$
Note that the order on the right does not matter, since $\phi(\gamma) \in \ker (p)$ is central (Lemma~\ref{central}); this also shows that $\phi\cdot \rho$ is a homomorphism.  Centrality of $\ker (p)$ also implies that this action descends to an action on  $\fQ_\Gamma (H)$, since conjugation is equivariant with respect to central multiplication.  (Note that the action carries closed sets to closed sets, and thus restricts to the polystable quotient $\fX_\Gamma(H)\subset \fQ_\Gamma (G)$.) 
The action becomes trivial after applying $p_*$, giving a commutative diagram
$$\xymatrix{ \fQ_\Gamma (H) \ar[r]^{\ol{p}_*} \ar[d] &\fQ_\Gamma (G) \\ \fQ_\Gamma (H)/ K. \ar@{-->}[ur]^\pi}$$

\begin{lemma} $\label{embedding}$ The induced map $\pi\co \fQ_\Gamma (H)/ K \to \fQ_\Gamma (G)$ is 
a homeomorphism onto its image.
\end{lemma}
\begin{proof}  

To prove injectivity, we must show that if $\ol{p}_* [\rho] =  \ol{p}_*[\psi]$ for two representations $\rho, \psi\co \Gamma\to H$, then there exists an element $k\in K$ such that $k \cdot [\rho] = [\psi]$.  

Since $\ol{p}_* [\rho] = \ol{p}_* [\psi]$, there exists $g\in G$ such that for each $\gamma\in \Gamma$ we have
$$ g p(\psi (\gamma)) g^{-1} = p (\rho(\gamma)).$$
Choose an element $h \in H$ satisfying $p (h) = g$.   Then for each $\gamma\in \Gamma$, we have
$$p ( h \psi (\gamma) h^{-1}) = g p (\psi (\gamma)) g^{-1} =  p (\rho (\gamma)),$$
and we define $k(\gamma)\in \ker (p)$ by 
$$k(\gamma) = h \psi (\gamma) h^{-1} \rho (\gamma^{-1}).$$
We claim that $k\co \Gamma\to \ker(p)$ is a homomorphism.  Indeed, since $k(\gamma)\in \ker (p)$ is always central, we have
\begin{eqnarray*}
k(\gamma) k(\eta) & = \left(h \psi (\gamma) h^{-1} \rho (\gamma^{-1}) \right) \left(h \psi (\eta) h^{-1} \rho (\eta^{-1}) \right)\\
&= h \psi (\gamma) h^{-1} \left(h \psi (\eta) h^{-1} \rho (\eta^{-1}) \right)\rho (\gamma^{-1})\\
&= h \psi (\gamma \eta) h^{-1} \rho ( (\gamma\eta)^{-1})
= k(\gamma \eta).
\end{eqnarray*}
Since $k \cdot [\rho] = [\psi]$, this proves injectivity.

We must check that $\pi$ is open as a map to its image.
We use the following elementary fact:
\begin{enumerate}[($\star$)]
\item Given maps $X\srm{\alpha} Y \srm{\beta} Z$, if $\alpha$ is surjective and $\beta\circ \alpha$ is open, then $\beta$ is open as well.
\end{enumerate}
Consider the commutative diagram
$$\xymatrix{ \Hom(\Gamma, H) \ar@{->>}[r] \ar[d]^{p_*} & \fQ_\Gamma (H)/ K \ar[d]^\pi\\
\Img (p_*) \ar[r]^-f & \Img (\ol{p}_*) = \Img (p_*)/G.}
$$
Since the top map is surjective, ($\star$) tells us that in order to check that $\pi$ is open, it suffices to show that $f\circ p_*$ is open.  By Lemmas~\ref{Goldman} and~\ref{triang},  $p_*$ is a covering map (onto its image), and all covering maps are open.
Next,  $f$ is open because it is the quotient map for a group action, and it follows that $f\circ p_*$ is open.
\end{proof}

\begin{lemma}$\label{prop-disc-lemma}$ The action of $K=\Hom(\Gamma, \ker (p))$ on $\fQ_\Gamma (H)$ is properly discontinuous.
\end{lemma}
\begin{proof}  
Let $h \co F_r \to \Gamma$ be a surjection, where $F_r$ is the free group on $r$ generators.  Then precomposition with $h$ induces an injection $K\injects \Hom(F_r, \ker(p)) = \ker(p)^r$, as well as an embedding
$$\fQ_\Gamma (H) \injects \fQ_{F_r} (H) = H^r/H.$$
The multiplication actions of $K$ on $\fQ_\Gamma (H)$ and of $\ker (p)^r$ on $H^r/H$ are compatible, so it suffices to check that the action of $\ker(p)^r$ on $H^r/H$ is properly discontinuous.

We begin by considering the case where $H = \wt{G}$ is the universal cover, and the map $p=q$ is the universal covering map in Proposition~\ref{univ-cover}.
By Proposition~\ref{univ-cover}  we know that the projection $\pi\co \ker (p)\to \bbF^k$ is injective with image a discrete subgroup of $\bbF^k$.  Letting $\epsilon = \min\{ |x| : x\in \pi(\ker(p)), x\neq 0\}$, where $|x|$ is the norm of $x$,
consider an open neighborhood of the form $U:=B_{\epsilon/2} (a)\cross V$, with $V$ open in $DG$, and $B_{\epsilon/2} (a)$ the open ball of radius $\epsilon/2$ around $a\in \bbF^k$.  Then each such $U$ is translated off of itself by each non-trivial element in $\ker (p)$.  Considering products of such neighborhoods, we see that the action of $\ker (p)^r$ on 
$\wt{G}^r$ is properly discontinuous.  Saturating these open sets with respect to conjugation, and noting that conjugation is trivial on the $\bbF^k$ coordinate in each factor, we conclude that $\pi_1 (G)^r$ acts properly discontinuously on $\wt{G}^r/\wt{G}$.

Now consider the general case of $p\co H\to G$, and let $q\co \wt{G}\to G$ be the universal covering homomorphism as above.  Then there exists a unique (universal) covering map $p'\co \wt{G}\to H$ satisfying $p' (e) = e$, and this map is automatically a homomorphism.  We have shown that $\ker(q)^r$ acts properly discontinuously on $\wt{G}^r/\wt{G}$, so the subgroup $\ker(p')^r$ acts properly discontinuously as well, and thus $(\ker(q)^r)/(\ker(p')^r) = \ker(p)^r$ acts properly discontinuously on 
$$(\wt{G}^r/\wt{G})/ (\ker(p')^r).$$
Finally,
$$(\wt{G}^r/\wt{G})/ (\ker(p')^r) \homeo H^r/H,$$
by applying Lemma~\ref{embedding} in the case $\Gamma = F_r$; note that $\Hom(\Gamma, H)$ is locally path connected by Lemma~\ref{triang}.
\end{proof}

Finally, we consider the image of $\ol{p}_* \co \fQ_\Gamma (H) \to \fQ_\Gamma (G)$.

\begin{lemma}$\label{image}$ The image $\ol{p}_* ( \fQ_\Gamma (H) )$ is a union of path components inside $\fQ_\Gamma (G)$.
\end{lemma}
\begin{proof} Since $\Hom(\Gamma, G)$ is triangulable, it is the disjoint union, topologically, of its path components $\{P_i\}_i$, and 
since $G$ is (path) connected, the action of  $G$ on $\Hom(\Gamma, G)$ preserves path components.  Hence the quotient space $\fQ_\Gamma (G)$ is the disjoint union, topologically, of $\{P_i/G\}_i$.  Now if $[\rho] = \ol{p}_* ([\psi])$ and $[\rho]$ is connected by a path in $\fQ_\Gamma (G)$ to $[\rho']$, then $\rho$ and $\rho'$ lie in the same path component $P_i$ of $\Hom(\Gamma, G)$.  By Lemma~\ref{Goldman}, the image of $p_*$ is a union of path components inside $\Hom(\Gamma, G)$, so $\rho\in \Img (p_*)$ implies $\rho'\in \Img (p_*)$ as well.  
\end{proof}

This completes the proof of Theorem~\ref{prop-discont} for the  map $\ol{p}_*\co\fQ_\Gamma (G) \to \fQ_\Gamma (H)$.

Now we consider the restriction of $\ol{p}_*$ to $\fX_\Gamma (H)$.  Since $\fX_\Gamma (H)$ is invariant under the group $K$ of deck transformations, the restriction of $\ol{p}_*$ to $\fX_\Gamma (H)$ is still a normal covering map with structure group $K$.   We need to prove (\ref{img}).

The moduli spaces $\fQ_\Gamma (G)$ and $\fQ_\Gamma (H)$ do not have good separation properties when $G$ and $H$ are non-compact, so we need some technical lemmas regarding closed points.  

\begin{lemma}$\label{closed-pts}$ Let $p\co X\to Y$ be a covering map between arbitrary topological spaces.  If $x\in X$ is a closed point, then $p(x)\in Y$ is also a closed point.
\end{lemma}

Note that in general, covering maps are not closed maps; for instance the usual covering $\exp\co \bbR\to S^1$ sends the closed set $\{n+1/n \,:\, n= 2, \ldots\}$ to a set with $1\in S^1$ as a limit point.

\begin{proof}  
Let $y\in Y$ be a point in the closure of $p(x)$, and let $U\subset Y$ be an open neighborhood of $y$ over which $p$ is trivial.   Since $y$ is in the closure of $p(x)$, we have $p(x) \in U$, and we claim that $p(x)$ is closed in the relative topology on $U$.  Triviality of $p$ over $U$ implies that $x\in p^{-1} (U)$ lies in a subset $U_1\subset p^{-1} (U)$ for which $p\co U_1\to U$ is a homeomorphism.  Since $x$ is closed in $X$, it is also closed in the relative topology on $U_1$, and since $p\co U_1\to U$ is a homeomorphism, we find that $p(x)$ is closed in the relative topology on $U$, as claimed.  
Now we can write $\{p(x)\} = C\cap U$ for some closed set $C\subset Y$, and since $y$ is in the closure of $p(x)$ we have $y\in C$.  But $y\in U$ by choice of $U$, so in fact $y\in C\cap U$ and hence $y = p(x)$. 
\end{proof}

\begin{lemma}$\label{closed}$ Let $X$ be an arbitrary topological space, and let $P \subset X$ be a union of path components of $X$.  If $x\in P$ is closed as a subset of $P$, then $x$ is also closed as a subset of $X$.
\end{lemma}
\begin{proof}  For any point $y\in X$, the closure $\ol{y}$ of $\{y\}$ in $X$ is path connected.  Indeed, say $z\in\ol{y}$, and consider  the function $f\co [0,1]\to \{y, z\}$ defined by $f(t) = y$ for $t<1$ and $f(1) = z$.
The set $\{z\}$ is \e{not} open in $\ol{y}$, so $f$ is a continuous map.  

Now say $x\in P$ is closed as a subset of $P$, that is, $\{x\} = C \cap P$ for some closed set $C\subset X$.  
From the above, we have $\ol{x} \subset P$, and $\ol{x} \subset C$ by definition, so $\ol{x} \subset P\cap C = \{x\}$.  Hence $x$ is closed in $X$.
\end{proof}

Lemma~\ref{closed-pts} shows that all points in $\ol{p}_* ( \fX_\Gamma (H))$ are closed inside $\ol{p}_* ( \fQ_\Gamma (H))$, and Lemmas~\ref{image} and~\ref{closed} show that in fact these points are closed in $\fQ_\Gamma (G)$.  Thus $\ol{p}_* ( \fX_\Gamma (H)) \subset \fX_\Gamma (G)$, and to prove (\ref{img}) it remains to check that
$$ \ol{p}_* (\fQ_\Gamma (H)) \cap  \fX_\Gamma (G) \subset \ol{p}_* ( \fX_\Gamma (H)).$$
Say $[\rho] \in  \ol{p}_* (\fQ_\Gamma (H)) \cap  \fX_\Gamma (G) $.  
Then $[\rho]$ is a closed point in the quotient $\fQ_\Gamma (G)$, so its fiber $(\ol{p}_*)^{-1} ([\rho])  \subset \fQ_\Gamma (H)$ is a closed subset.  But $\ol{p}_*$ is a covering map, so this inverse image has the discrete topology, and hence each point in $(\ol{p}_*)^{-1} ([\rho])$ is actually closed in $\fQ_\Gamma (H)$.   This proves  (\ref{img}).

Finally, we must show that $\ol{p}_* ( \fX_\Gamma (H))$ is a union of path components inside $\fX_\Gamma (G)$.  If $[\rho] \in \ol{p}_* ( \fX_\Gamma (H))$ and $[\psi] \in  \fX_\Gamma (G)$ is connected to $[\rho]$ by a path, then
$[\psi]\in \ol{p}_* ( \fQ_\Gamma (H))$ by Lemma~\ref{image}.  So $[\psi] = p_* [\psi']$ for some $[\psi']$, and we just need to check that $[\psi']$ is closed.  But we have just seen that the pre-image of a closed point consists only of closed points.
This completes the proof of Proposition~\ref{prop-discont}. 
\end{proof}

\begin{corollary}  $\label{pi}$
Let $G$ be either a connected reductive $\C$-group or a connected compact Lie group, and assume that $\pi_1 (G)$ is torsion-free.  Let $H \srt{p} G$ be a covering homomorphism. 
Then the maps
 $$\fQ_{\Gamma} (H) \srm{q_*} \fQ_\Gamma (G) \textrm{ and } 
 \fX_\Gamma (H)\srm{q_*} \fX_\Gamma (G)$$
 induce isomorphisms on homotopy groups $\pi_k$ for all $k\geqs 2$ {\rm (}and all compatible basepoints{\rm )}.
 \end{corollary}

\begin{question}
Do Theorem~\ref{prop-discont} and Corollary~\ref{pi} hold for coverings of connected {\it real} reductive algebraic groups {\rm (}or even more general connected Lie groups{\rm )} with $\pi_1 (DG)$ trivial?
\end{question}

 We end this section with an example  to show that Theorem~\ref{prop-discont}  does not extend to all covering maps between compact, connected Lie groups.  This example will also be important in Section~\ref{stable-sec}.

\begin{example}$\label{SU(2)}$
The moduli space $\fX_{\bbZ^2} (\SUn(2))$ is homeomorphic to $S^2$.  This can be seen from a variety of perspectives.  
For instance, the homeomorphism $(\SUn(2)\cross \SUn(2))/\SUn(2) \isom B^3$ {\rm (}where $\SUn(2)$ acts by conjugation in each coordinate and $B^3$ is the unit ball in $\bbR^3)$ carries the subspace $\fX_{\bbZ^2} (\SUn(2))$ to the boundary sphere {\rm (}see \cite{MoSh}, \cite{JW}, \cite{brCo}, \cite{FlLaFree}{\rm )} 
Alternatively,~\cite[Proposition 6.5]{Adem-Cohen-Gomez} uses the Abel--Jacobi map to show that $\fX_{\bbZ^2} (\SUn(2)) \isom \bbC \textrm{P}^{1}$.  

The projection map  $\SUn(2) \to \SUn(2)/\{\pm I\} = \mathrm{P}\SUn(2)$ is a covering homomorphism with structure group $\bbZ/2\bbZ$.  
A direct generalization of Proposition~\ref{prop-discont} would say that the induced map
$\fX_{\bbZ^2} (\SUn(2)) \to \fX_{\bbZ^2} (\mathrm{P}\SUn(2))$ is a covering with structure group $\Hom(\bbZ^2, \bbZ/2\bbZ) = (\bbZ/2\bbZ)^2$, but this is impossible; in fact there is no free action of $(\bbZ/2\bbZ)^2$ on $S^2$, since the quotient space for such an action would be a closed surface with fundamental group $(\bbZ/2\bbZ)^2$.  We note that in fact $\fX_{\bbZ^2} (\mathrm{P}\SUn(2)) \isom S^2$, as can be checked by direct computation.  
\end{example}

\section{Universal covers and fundamental groups of moduli spaces}$\label{univ-cover-sec}$

In this section we show that in certain cases, the covering maps constructed in the previous sections are in fact universal covering maps, leading to new results on fundamental groups.  Throughout this section, $\wt{G}\srt{q} G$ will denote the universal covering homomorphism for the Lie group $G$, and $DG\leqs G$ will denote the commutator subgroup.

In this section we will focus on \e{exponent-canceling} groups.

\begin{definition}  Given a set $S$, let $F_S$ denote the free group on $S$.  We say that a word $R\in F_S$ is exponent-canceling if $R$ maps to the identity in the free Abelian group on $S$ {\rm (}in other words, $R$ lies in the commutator subgroup of $F_S${\rm )}.   

Let $\Gamma$ be a finitely generated discrete group.  If $\Gamma$ admits a presentation 
$$\langle \gamma_1, \ldots, \gamma_r \,|\, \{R_i\}_i\rangle$$
in which each word $R_i$ is exponent-canceling, then we say that $\Gamma$ is exponent-canceling.
We will refer to a generating set in an exponent-canceling group $\Gamma$ as \e{standard} if there is a presentation of $\Gamma$ with these generators in which all relations are exponent-canceling.  The number $r$ of generators in a standard presentation will be called the \e{rank} of $\Gamma${\rm;} note the Abelianization of $\Gamma$ is 
always of the form $\bbZ^r$.
\end{definition} 

Examples of exponent-canceling groups include   free groups, free Abelian groups, right-angled Artin groups, and fundamental groups of closed Riemann surfaces, as well as the universal central extensions of surface groups considered in~\cite{BGG}.
We note that the class of exponent-canceling groups is closed under free and direct products.

\begin{proposition}$\label{surj}$ If $\Gamma$ is an exponent-canceling group of rank $r$, then the covering maps considered in Theorem~\ref{prop-discont} are always surjective, and the structure group $\Hom(\Gamma, \ker(p))$ is simply $\ker (p)^r$.
\end{proposition}

\begin{proof} The isomorphism $\Hom(\Gamma, \ker(p))\isom \ker (p)^r$ follows from the facts that  $\ker (p)$ is Abelian and the Abelianization of $\Gamma$ is $\bbZ^r$.

Now we consider surjectivity.  Let $G$ be as in Theorem~\ref{prop-discont}.  Since the universal covering homomorphism $q\co \wt{G}\to G$ factors through every covering $p\co H\to G$, it suffices to check surjectivity of $\ol{q}_*$.  (Note that by Equation (\ref{img}) in Theorem~\ref{prop-discont}, surjectivity on the full moduli spaces implies surjectivity on the polystable quotients.)

     Take any element $[\rho]\in \fQ_\Gamma(G)$.  Then $[\rho]$ lifts to $\rho=(g_1, \ldots, g_r)\in \Hom(\Gamma, G)$, and we can lift each $g_i$ to $\tilde{g}_i\in \wt{G}$.  By Proposition~\ref{univ-cover}, for each $i$ we may write $\tilde{g}_i = (\tilde{t}_i, g_i')\in \bbF^k \cross DG$ (where $\bbF = \bbR$ or $\bbC$ depending on whether $G$ is compact or complex).
 Now $g_i = q(\tilde{t}_i) q(g'_i) = q(\tilde{t}_i) g'_i$, and $t_i := q(\tilde{t}_i)$ is central in $G$.  For each relation $R(\gamma_1,\ldots ,\gamma_r)$ in $\Gamma$, we have $R(g_1,\ldots ,g_r)=e$ where $e$ is the identity.  Thus, 
$$e=R(t_1g_1',\ldots ,t_rg_r')=R(t_1, \ldots, t_r) R(g_1',\ldots, g_r') = R(g_1',\ldots, g_r')$$ 
since $t_1,\ldots ,t_r$ are central in $G$ and $R$ is exponent-canceling.  Similarly, 
$$R(\tilde{g}_1, \ldots, \tilde{g}_r)=R((\tilde{t_1},g_1'),\ldots ,(\tilde{t_r},g_r'))=(R(\tilde{t}_1,\ldots ,\tilde{t}_r),R(g_1', \ldots, g_r'))=e.$$  
Therefore, $\tilde{\rho}=(\tilde{g}_1, \ldots,\tilde{g}_r)$ satisfies the relations of $\Gamma$ and hence $\tilde{\rho} \in \Hom(\Gamma, \wt{G})$.  Since $q_* (\tilde{\rho}) = \rho$, we conclude that $q_*$ is surjective, and surjectivity of $\ol{q}_*$ follows.  
 \end{proof}

\begin{lemma}$\label{heqs}$
Let $G$ be either a connected compact Lie group, or a connected reductive $\bbC$--group.  Assume that $\pi_1 (G)$ is torsion-free
and $\Gamma$ is a finitely generated group.  Then there are homotopy equivalences 
\begin{equation}\label{heq}\Hom(\Gamma, \wt{G}) \heq \Hom (\Gamma, DG) \textrm{ and } \fX_{\Gamma}(\wt{G}) \heq  \fX_{\Gamma}(DG).\end{equation}
\end{lemma}
\begin{proof} By Proposition~\ref{univ-cover}, $\wt{G} \isom \bbF^k \cross DG$ for some $k\geqs 0$, and hence
$$\Hom (\Gamma, \wt{G}) \isom \Hom(\Gamma, \bbF^k)\cross \Hom(\Gamma, DG)  \isom \bbF^{kr} \cross  \Hom(\Gamma, DG) \heq  \Hom(\Gamma, DG),$$ 
where $r$ is the rank of the Abelianization of $\Gamma$.  Since $\bbF^{k}$ is central in $\wt{G}$, the conjugation action of $\wt{G}$ reduces to that of $DG$, so we have $\fQ_{\Gamma}(\wt{G})\cong \bbF^{rk}\times \fQ_{\Gamma}(DG)$.
This homeomorphism restricts to a homeomorphism $\fX_{\Gamma}(\wt{G})\cong \bbF^{rk}\times \fX_{\Gamma}(DG)$ between the subsets of closed points, and $\bbF^{rk}\times \fX_{\Gamma}(DG) \heq  \fX_{\Gamma}(DG)$.
\end{proof}

\begin{proposition}$\label{Hom}$
Let $G$ be either a connected compact Lie group, or a connected reductive $\bbC$--group.  Assume that $\pi_1 (G)$ is torsion-free
and $\Gamma$ is exponent-canceling of rank $r$.   If $\fX_\Gamma (DG)$ is simply connected, then the map
$$\fX_{\Gamma} (\wt{G}) \srm{q_*} \fX_\Gamma (G)$$
is a universal covering map.  Consequently, $\pi_1 (\fX_\Gamma (G)) \isom \pi_1 (G)^r$.
\end{proposition}
\begin{proof} By Theorem~\ref{prop-discont}, we know that $q_*$ is a covering map with structure group $\pi_1 (G)^r$, and since $\fX_{\Gamma} (\wt{G})\heq \fX_{\Gamma}(DG)$ by Lemma~\ref{heqs}, the result follows.
\end{proof}

As we will explain below, $\fX_\Gamma (DG)$ is simply connected whenever the representation space
$\Hom(\Gamma, DG)$ is simply connected (see Section~\ref{lifting-sec}).

\begin{remark} Since we have assumed that $\pi_1 (G)$ is torsion-free, it is  isomorphic to $\bbZ^{k}$ for some $k\geqs 0$.  Thus the group $\pi_1 (\fX_\Gamma (G)) \isom \pi_1 (G)^r$ in Proposition~\ref{Hom} is free Abelian as well.
We have $G=\textrm{Rad}(G)DG$ {\rm (}see for example~\cite[Theorem 1.29]{milneRG}{\rm )}, and the projection from $G$ to the torus $G/DG = \textrm{Rad}(G)/(\textrm{Rad}(G)\cap DG)$ induces an isomorphism on $\pi_1$.  Since $\textrm{Rad}(G)\cap DG$ is central it is finite, so we find that $k = \dim_\bbF (\textrm{Rad}(G)/(\textrm{Rad}(G)\cap DG)) = \dim_\bbF (\textrm{Rad}(G))$ {\rm (}where $\bbF = \bbR$ or $\bbC$ depending on whether $G$ is real or complex{\rm )}.
The rank of $\pi_1 (G)$ is also the same as the number $k$ in the decomposition $\wt{G} \isom \bbF^k \cross DG$ {\rm (}Proposition~\ref{univ-cover}{\rm )}
since $q\co \bbF^k \cross DG\to G$ restricts to a covering $\bbF^k \to \textrm{Rad}(G)$.
\end{remark}
 
\subsection{Free groups}

\begin{corollary}$\label{free-univ-cover}$
Let $\Gamma = F_r$ be the free group of rank $r$.  Let $G$ be either a connected compact Lie group, or a connected reductive $\bbC$--group.  If $\pi_1 (G)$ is torsion-free, then the map $\fX_\Gamma (\wt{G}) \srm{q_*} \fX_\Gamma (G)$ is the universal covering map.  In particular, $\pi_1  (\fX_\Gamma (G)) \isom \pi_1 (G)^r.$
\end{corollary}

\begin{proof} 
It suffices, in light of Proposition~\ref{Hom}, to prove that $\fX_{F_r}(DG)$ is simply connected.   Since $\pi_1(G)$ is torsion-free, $\pi_1(DG)=1$ (Proposition~\ref{univ-cover}).  In \cite{FlLaFree}, it is shown (using methods similar to those below) that for both the complex reductive and compact cases, $\fX_{F_r}(DG)$ is simply connected whenever $DG$ is simply connected. 
\end{proof}

\begin{remark}
In \cite{BiLa}, without the assumption that $\pi_1(G)$ is torsion-free, part of this result is extended.  In particular, it is shown that $\pi_1  (\fX_\Gamma (G))\isom \pi_1 (G/DG)^r$ when $G$ is either a connected reductive $\bbC$--groups or a connected compact Lie group {\rm (}this is consistent with our result given Remark~\ref{torsion-rmk}{\rm )}.
\end{remark}

By Casimiro et. al.~\cite[Theorem 4.7]{CFLO}, the character varieties $\fX_{F_r} (G)$ and  $\fX_{F_r} (K)$ are homotopy equivalent whenever $G$ is a real $K$--reductive algebraic group (see \cite[Definition 2.1]{CFLO}).
Here $K\leqs G$ is a maximal compact subgroup.  Since $\pi_1 (G) \isom \pi_1 (K)$, we obtain the following consequence of Corollary~\ref{free-univ-cover}.

\begin{corollary}$\label{real-red-free}$ Let $G$ be a connected real reductive Lie group with $\pi_1 (G)$ torsion-free.
Then $\pi_1  (\fX_{F_r} (G)) \isom \pi_1 (G)^r$.
\end{corollary}

\subsection{Free Abelian groups}$\label{Abelian-sec}$

We say that a semisimple Lie group is \e{orthogonal-free} if it is a direct product of simply connected groups of type $A_n$ or type $C_n$.

\begin{corollary}$\label{free-ab-univ-cover}$
Let $\Gamma = \bbZ^r$ be the free Abelian group of rank $r$.  Let $G$ be either a connected compact Lie group, or a connected reductive $\bbC$--group.  Assume either $(a)$ that $r\geqs 3$ and $DG$ is orthogonal-free, or  $(b)$ that $r=1,2$ and $\pi_1 (G)$ is torsion-free.  Then 
the map $\fX_\Gamma (\wt{G}) \srm{q_*} \fX_\Gamma (G)$ is a  universal covering map.  In particular, $\pi_1 (\fX_\Gamma(G)) \isom \pi_1 (G)^r.$
\end{corollary}
\begin{proof}
As in the case of free groups, 
by Proposition~\ref{Hom} it  suffices to prove that $\fX_\Gamma (DG)$ is simply connected.
We remind the reader that in the compact cases $\fQ_\Gamma(G)=\fX_\Gamma(G)$ since compact orbits are always closed.  
We begin by noting that if $DG$ is orthogonal-free, it is a product of simply connected groups, and hence simply connected itself.  Thus by Proposition~\ref{torsion-free}, $\pi_1 (G)$ is torsion-free in this case as well.

It is known that $\Hom(\bbZ^r, H)$ is path-connected for any $r\geqs 3$ and $H$ orthogonal-free, or $r=2$ and $H$ simply connected, or $r=1$.  The compact case is due to  Kac--Smilga~\cite{Kac-Smilga} (for $r > 3$, see \cite[Theorem D1]{Kac-Smilga}).  Another proof in the compact, orthogonal-free case is given in~\cite[Corollary 2.4]{Adem-Cohen}, and the reductive case is due to Pettet--Souto~\cite{PeSo}.  Therefore, $\fQ_{\bbZ^r} (DG)$ is path-connected in the  cases of interest.  
By results in Florentino--Lawton~\cite{FlLaAbelian}, $\fX_{\bbZ^r} (DG)$ is also path-connected in these cases 
(these results apply since $G$ is assumed connected).
As explained in Corollary 5.20 of \cite{FlLaAbelian}, simple-connectivity of $\fX_{\bbZ^r} (DG)$ now follows from 
work of Gomez--Pettet--Souto~\cite{GPS} in the compact case, and the same holds in the complex case by Theorem 1.1 in \cite{FlLaAbelian}.

\end{proof}

\begin{remark} As explained in \cite{FlLaAbelian}, using results from \cite{Kac-Smilga}, the above cases are the  only cases where the  $\fX_{\bbZ^r} (G)$ is connected when G is semisimple and $r\geqs 3$.
\end{remark}

As in the previous section, we can extend our computation of fundamental groups to the real reductive case, using the deformation retraction in Florentino--Lawton~\cite[Theorem 1.1]{FlLaAbelian}.

\begin{corollary}\label{real-red-abelian} Let $G$ be a real reductive Lie group with  maximal compact subgroup $K \leqs G$, and assume either that  $DK$ is orthogonal-free, or that $r=1,2$ and $\pi_1 (K)$ is torsion-free. Then 
$\pi_1 (\fX_{\bbZ^r} (G)) \isom \pi_1 (G)^r$.
\end{corollary}

\subsection{Path-lifting for reductive actions}$\label{lifting-sec}$

It is possible to prove the previous results without using the deformation retraction results \cite[Theorem 1.1]{FlLaFree} and \cite[Theorem 1.1]{FlLaAbelian}.
In Bredon~\cite[Theorem II.6.2]{Bredon-transf} it is shown that for any Hausdorff space $X$ with a compact Lie group $K$ acting, the natural map $X\to X/K$ has the path-lifting property.  In \cite{KPR} this is generalized to affine varieties with complex reductive group actions, and in \cite{BiLaRa} it is generalized to real reductive group actions.  We provide a proof here to make the article more self-contained; this will be used in the next section to study surface groups.

\begin{lemma}\label{pathlifting}
Let $X$ be a complex affine variety, and let $G$ be a connected reductive $\C$-group.  If $G$ acts rationally on $X$, then the map $X\to X\quot G$ has the path-lifting property.
\end{lemma}

\begin{proof}
By Kempf-Ness~\cite{KN}, there exists a real algebraic subset $V\subset X$ such that $X\quot G$ is homeomorphic to $V/K$, where $K$ is a maximal compact subgroup of $G$ (see also \cite{Sch1}).  Moreover, the natural diagram $$\xymatrix{V \ar@{->>}[r] \ar@{^{(}->>}[d] & V/K\ar[d]^{\cong} \\ X \ar[r] & X\quot G}$$ commutes. So we need only verify that $V\to V/K$ has the path-lifting property.  However, since $V$ is algebraic, it satisfies the conditions in \cite{MY} for there to be a slice at each point.  And as shown in \cite[page 91]{Bredon-transf}, this implies that there exist a lift of each path in $V/K$ to $V$.
\end{proof}

The main results in \cite{PeSo} and \cite{GPS} imply that $\Hom(\bbZ^r,G)$ is simply connected whenever $G$ is simply connected {\rm (}compact or complex reductive{\rm )}.  Therefore, by Lemma~\ref{pathlifting} and its compact analogue in \cite{Bredon-transf}, $\fX_{\Z^r}(DG)$ is simply connected in the cases considered in Corollary~\ref{free-ab-univ-cover} since $DG$ is path-connected {\rm (}note that the conjugation action of $G$ is rational since $G$ is an algebraic group{\rm )}.  Here is the proof:  we can lift any loop based at $[\rho]$ in the quotient to a path with its ends in the path-connected fiber $DG/\mathrm{Stab}([\rho])$; thus the quotient map is $\pi_1$-surjective, and the total space is simply connected. This gives a different proof of Corollary~\ref{free-ab-univ-cover} {\rm (}which can be adapted  to the free group case as well{\rm )}.

\subsection{Surface groups}$\label{surface-sec}$

Let $M^g$ be the closed orientable surface of genus $g\geq 1$.  In this section we focus on surface groups $\Gamma^g := \pi_1 M^g$, which are exponent-canceling of rank $2g$.  Note that the $g=1$ case was handled in the last subsection.

In \cite{Li-surface-groups}, it is shown that when $G$ is a connected semisimple $\C$-group and $g>1$, there is a bijection 
$\pi_0(\Hom(\Gamma^g)) = \pi_1 (G)$, and moreover, when $\pi_1(G)=1$ it is shown that $\Hom(\Gamma^g,G)$ is simply connected.  Therefore, by Lemma~\ref{pathlifting} and \cite[Lemma 4.11]{FlLaAbelian}, we conclude:

\begin{theorem}\label{wowhowdopeoplenotknowthis}
$\fX_{\Gamma^g}(G)$ is  simply connected for $g > 1$ and $G$ a simply connected semisimple $\C$-group.
\end{theorem}

In \cite{Ho-Liu-ctd-comp-II, A-B}, it is shown that 
\begin{equation} \label{HL}  \pi_0(\Hom(\Gamma^g,G))\cong \pi_0(\fX_{\Gamma^g}(G))\cong \pi_1(DG)\end{equation}
for any $g>0$ and any compact connected Lie group.  In Appendix~\ref{reductivecomponents}, Ho and Liu generalize this same result to the case when $G$ is a connected reductive $\C$-group.

\begin{corollary}\label{univ-cover-surface-reductive}
Let $G$ be a connected reductive $\C$-group, and assume $\pi_1(G)$ is torsion-free.  Then $\fX_{\Gamma^g}(\wt{G})\srm{q_*}\fX_{\Gamma^g}(G)$ is the universal cover.  In particular, $\pi_1(\fX_{\Gamma^g}(G))\isom \pi_1(G)^{2g}$. 
\end{corollary}

\begin{proof}
The case $g=1$ was handled in Section~\ref{Abelian-sec}, so we assume $g\geq 2$.  As before, it suffices to check that $\fX_{\Gamma^g}(DG)$ is simply connected.  By Proposition~\ref{univ-cover}, we know that $\pi_1 (DG) = 1$, so   $\fX_{\Gamma^g}(DG)$ is  simply connected by Theorem~\ref{wowhowdopeoplenotknowthis}.  
\end{proof}

\begin{corollary}$\label{univ-cover-surface}$  Let $G=\Un(n)$.  The covering map 
$\fX_{\Gamma^g} (\wt{G}) \srm{q_*} \fX_{\Gamma^g}(G)$ is the universal covering map.  Consequently,
$\fX_{\Gamma^g} (\SUn(n))$ is simply connected, and $\pi_1\left(\fX_{\Gamma^g} (\Un(n))\right)\cong \Z^{2g}.$
\end{corollary}

\begin{proof}  We have $\wt{G} \isom \bbR\cross \SUn(n)$.  
Since $DG = \SUn(n)$ is connected, it follows from \cite{Ho-Liu-ctd-comp-II} (see  (\ref{HL})) that $\Hom(\Gamma^g, \SUn(n))$ and $\Hom(\Gamma^g,\Un(n))$, and consequently $\fX_{\Gamma^g} (\SUn(n))$ and $\fX_{\Gamma^g} (\Un(n))$, are path-connected.  
By  Lemma~\ref{heqs}, we conclude that $\fX_{\Gamma^g} (\bbR\cross \SUn(n))$ is path-connected as well.

Lemma~\ref{Goldman} tells us that $q_*\co \Hom(\Gamma^g, \bbR\cross \SUn(n)) \to \Hom (\Gamma^g, \Un(n))$ is a covering map whose image is a union of path components; since the base space is in fact path-connected, we find that $q_*$ is surjective.  Again, by Lemma~\ref{Goldman}, the structure group of this covering is $\Hom(\Gamma, \ker(\pi)) = \Hom(\Gamma^g, \bbZ) \isom \bbZ^{2g}$.   
This gives a short exact sequence of groups
$$1\maps \pi_1 (\Hom(\Gamma^g, \bbR\cross \SUn(n))) \maps \pi_1 (\Hom (\Gamma^g, \Un(n))) \srm{p} \bbZ^{2g}\maps 1.$$
In \cite[Proposition 5.1]{Ho-Liu-Ramras} it was shown that the determinant map 
$$\det \co \Hom(\Gamma^g, \Un(n)) \maps \Hom(\Gamma^g, U(1)) \isom (S^1)^{2g}$$ 
induces an isomorphism on $\pi_1$, so $ \pi_1 (\Hom (\Gamma^g, \Un(n))) \isom \bbZ^{2g}$.
Every surjection between free Abelian groups of the same rank is an isomorphism, so the map $p$ in the above exact sequence is an isomorphism and $\ker (p) = \pi_1 (\Hom(\Gamma^g, \bbR\cross \SUn(n)))$ is trivial.
By Lemma~\ref{heqs}, it follows that $ \Hom(\Gamma^g, \SUn(n))$ is simply connected.  Using path-lifting as in Section~\ref{lifting-sec} (or \cite[Corollary II.6.3]{Bredon-transf}) we conclude that $\fX_{\Gamma^g} (\SUn (n))$ is simply connected as well. 
Proposition~\ref{Hom} completes the proof.
\end{proof}

\begin{remark}
Biswas and Lawton also obtain Corollary~\ref{univ-cover-surface} and the $\GL(n,\C)$ case of Corollary~\ref{univ-cover-surface-reductive} in \cite{BiLa}.
\end{remark}

It is natural to ask whether Corollary~\ref{univ-cover-surface} extends to other connected compact Lie groups.  We offer a partial result in this direction.

\begin{corollary}  Let $G$ be a connected compact  Lie group with
$\pi_1 (G)$   torsion-free.  Then $\fX_{\Gamma^g} (\wt{G})$ is path connected, and there is a  short exact sequence of groups
$$1\maps \pi_1 (\fX_{\Gamma^g} (\wt{G})) \maps \pi_1 ( \fX_{\Gamma^g} (G)) \maps \pi_1 (G)^{2g}\maps 1.$$
\end{corollary}

\begin{proof} 
By Theorem~\ref{prop-discont}, we have the covering map $\fX_{\Gamma^g} (\wt{G}) \to \fX_{\Gamma^g} (G)$ with structure group $\pi_1 (G)^{2g}$, so it suffices to check that $\fX_{\Gamma^g} (\wt{G})$ is path-connected.  By Proposition~\ref{Hom}, we just need to check that $\fX_{\Gamma^g} (DG)$ is path-connected, but this follows from (\ref{HL}) since $DG$ is simply connected. 
\end{proof}

\begin{remark}
Let $G$ be a connected reductive $\C$-group, and $K$ a maximal compact subgroup.  In \cite{FlLaAbelian} and \cite{FlLaFree} it is shown that $\fX_{\Gamma}(G)$ and $\fX_{\Gamma}(K)$ are homotopic when $\Gamma$ is a finitely generated free group or a finitely generated Abelian group.  However, examples when $\fX_{\Gamma^g}(G)$ and $\fX_{\Gamma^g}(K)$ are not homotopic can be ascertained from Poincar\'e polynomial computations {\rm (}see \cite{Hitchin}, \cite{A-B}, \cite{DaWeWi}$);$ and in \cite{BiFl} it is shown they are never homotopic for $g\geq 2$.  This paper, however,  provides evidence that they are $\pi_k$-isomorphic for $k=0,1$.
\end{remark}

We end this section with some interesting questions.

\begin{question}
To what extent can we remove the assumption that $\pi_1(G)$ is torsion-free?  Note that because of Example~\ref{SU(2)}, a direct generalization of Proposition~\ref{prop-discont} is not possible.
\end{question}

\begin{question}\label{q2}
Can we generalize Corollary~\ref{univ-cover-surface} to all connected compact Lie groups whose derived subgroup is simply connected?
\end{question}

We note that Question~\ref{q2} can in fact be addressed using the methods in~\cite{BiLaRa}.

\begin{question} As discussed in the introduction, Theorem~\ref{univ-cover-surface-reductive} {\rm (}for $G = \GLn (n)${\rm )} extends~\cite[Theorem 4.4(1)]{BGG} to the degree zero component of $\fX_{\widehat{\Gamma^g}} (\GLn (n))$.
Can our methods  be used to compute fundamental groups for all components of this space?  Note that $\widehat{\Gamma^g}$ is exponent-canceling.  
\end{question}

\begin{question}What can be said when $\Gamma$ is a non-orientable surface group?  Note that since exponent-canceling groups have torsion-free Abelianizations, non-orientable surface groups are not exponent-canceling. 
\end{question}

\section{Stable moduli spaces of surface group representations}$\label{stable-sec}$

We now use Corollaries~\ref{univ-cover-surface-reductive} and~\ref{univ-cover-surface} to study the \e{stable moduli spaces} of surface group representations
$$\fX_{\Gamma^g} (\SUn) := \colim_{n\to \infty} \fX_{\Gamma^g} (\SUn(n)) \textrm{ and } \fX_{\Gamma^g} (\GLn) := \colim_{n\to \infty} \fX_{\Gamma^g} (\GLn(n)).$$
The colimits are taken with respect to the inclusions sending $A$ to $A \oplus [1]$ (the block sum of $A$ with the $1\cross 1$ identity matrix).

The unitary stable moduli space $\fX_{\Gamma^g} (U)$ is defined similarly, and
its homotopy type was calculated by the second author using results from Yang--Mills theory and stable homotopy theory~\cite{Ramras-stable-moduli}.

\begin{theorem}[Ramras] \label{U(n)-thm} There is a homotopy equivalence
$$\fX_{\Gamma^g} (\Un) \heq (S^1)^{2g} \cross \bbC \textrm{P}^\infty.$$
Hence $\pi_1 (\fX_{\Gamma^g} (\Un) ) \isom \bbZ^{2g}$, $\pi_2 (\fX_{\Gamma^g} (\Un) )\isom \bbZ$, and $\pi_*(\fX_{\Gamma^g} (\Un) )  = 0$ for $*\neq 1, 2$.
\end{theorem}

In fact, the same article also shows that the maps
$$(S^1)^{2g} \isom \fX_{\Gamma^g} (\Un(1)) \maps \fX_{\Gamma^g} (\Un) \maps \fX_{\Gamma^g} (\Un(1)) \isom (S^1)^{2g}  $$
induced by the inclusions $\Un(1)\to \Un(n)$ and the determinant maps $\Un(n) \to \Un(1)$ (respectively) each induce an isomorphism on fundamental groups.  We can now extend this result.

\begin{proposition} The functor $\pi_1 \left( \fX_{\Gamma^g} (-)\right)$ applied to the diagram
$$\xymatrix{ \Un(1) \ar@{^{(}->}[r] \ar@{^{(}->}[d] &\Un (n) \ar@{^{(}->}[r] \ar@{^{(}->}[d] &\Un (n+1)  \ar@{^{(}->}[d] \ar[r]^{\textrm{det}} &\Un(1)  \ar@{^{(}->}[d] \\
 \GLn(1) \ar@{^{(}->}[r]  &\GLn (n) \ar@{^{(}->}[r]  & \GLn (n+1)   \ar[r]^{\textrm{det}} &\GLn(1)}
$$
results in a diagram of isomorphisms between free Abelian groups of rank $2g$.  Consequently, $\pi_1 \left(\fX_{\Gamma^g} (\GLn)\right) \isom \bbZ^{2g}$.
\end{proposition}
\begin{proof} By Corollaries~\ref{univ-cover-surface-reductive} and~\ref{univ-cover-surface}, the fundamental groups of all of these character varieties are free Abelian of rank $2g$, and the determinant maps split the inclusions $\GLn(1) \injects \GLn(m)$ and $\Un(1) \injects \Un(m)$. 
Since every surjection $\bbZ^{2g}\to \bbZ^{2g}$ is an isomorphism, the result follows.  For the final statement, note that the maps 
$\fX_{\Gamma^g} (\GLn (n))\to \fX_{\Gamma^g} (\GLn (n+1))$ are embeddings of Hausdorff spaces, 
so $\pi_1 ( \fX_{\Gamma^g} (\GLn) ) \isom \colim_n \pi_1 ( \fX_{\Gamma^g} (\GLn (n)) )$.
\end{proof}

Using Corollary~\ref{univ-cover-surface}, we can also describe a generator of $\pi_2 (\fX_{\Gamma^g} (U)) \isom \bbZ$.  We will need to use the degree one ``pinch" maps
$p_g \co M^g \to M^1 = S^1\cross S^1$.  
These maps are characterized up to homotopy by their behavior on $\pi_1$, and in terms of the standard presentation $\Gamma^g = \langle a_1, b_1 \ldots, a_g, b_g \,|\, [a_1, b_1] \cdots [a_g, b_g] = e\rangle$, 
we have 
$(p_g)_* (a_1) = a_1$, $(p_g)_* (b_1) = b_1$, and $(p_g)_* (a_i) = (p_g)_* (b_i) = e$ for $i>1$ (on the right-hand side of these equations, $a_1$ and $b_1$ are the generators of $\Gamma^1 \isom \bbZ^2$).

\begin{theorem}$\label{stable-moduli-thm}$
There is a homotopy equivalence
$$\fX_{\Gamma^g} (\SUn) \heq \bbC \textrm{P}^\infty.$$
Moreover, the natural maps
$$\fX_{\bbZ^2} (\SUn(2)) \maps \fX_{\bbZ^2} (\SUn) \xmaps{p_g^*}  \fX_{\Gamma^g} (\SUn) \maps \fX_{\Gamma^g} (\Un)$$
each induce an isomorphism on $\pi_2$.
{\rm (}The first  and last maps are induced by the inclusions $\SUn(2)\injects \SUn(n)$ and $\SUn(n)\injects \Un(n)$, respectively.{\rm )}
\end{theorem}

Since $\fX_{\Gamma^g} (\SUn(2))$ is homeomorphic to $S^2$ (Example~\ref{SU(2)}), the above maps provide explicit generators for $\pi_2 (\fX_{\Gamma^g} (\SUn))$ and $\pi_2 (\fX_{\Gamma^g} (\Un))$.

We will need a lemma in the proof, and first we fix some notation.  Let $\SP^n (X)$ denote the $n$th symmetric product of the space $X$, namely the quotient of $X^n$ by the natural permutation action of the symmetric group $\Sigma_n$.  Given a basepoint $x\in X$, there are stabilization maps $\SP^n (X) \to \SP^{n+1} (X)$ that send $[(x_1, \ldots, x_n)]$ to $[(x_1, \ldots, x_n, x)]$.  When $X= M^g$ (equipped with a complex structure), the space $\SP^n (M^g)$ is a smooth manifold, which we view as the space of effective divisors on $M^g$ of total degree $n$.

Let  $J(M^g) \isom (S^1)^{2g}$ denote the Jacobian of $M^g$. 
The Abel--Jacobi map can be viewed as a smooth mapping $\mu = \mu_n \co \SP^n (M^g) \to J(M^g)$ (see~\cite{ACGH}, for instance).  Its definition relies on a choice of basepoint $m\in M^g$, and the corresponding stabilization map $\SP^n (M^g) \to \SP^{n+1} (M^g)$ covers the identity map on $J(M^g)$.
The following result is presumably well-known, but we have not found a discussion of these stabilization maps in the literature.

\begin{lemma}$\label{AJ}$ Assuming $n\geqs 2g$,
the Abel--Jacobi map
$\mu\co \SP^n (M^g) \to J(M^g)$ is a fiber bundle with fibers homeomorphic to $\bbC \textrm{P}^{n-g}$.  
The stabilization maps $\SP^n (M^g) \to \SP^{n+1} (M^g)$ restrict to the usual inclusions $\bbC \textrm{P}^{n-g} \injects \bbC \textrm{P}^{n+1-g}$.
\end{lemma}  
\begin{proof}   

The ideas involved are classical, and we refer to~\cite{ACGH} for background and details.  As pointed out in~\cite[Section 6]{Schwarzenberger}, the fact that $\mu$ is a fiber bundle follows from the general results in~\cite{Mattuck-picard}.

To describe the fibers of $\mu$, let $\mathcal{M} (M^g)$ denote the complex vector space of meromorphic functions $f\co M^g\to \bbC$.
  Abel's Theorem states that the fiber $\mu^{-1} (x)$ of $\mu$ over a point $x=\mu([c_1, \ldots, c_n])$ is the set
$$ \{[d_1, \ldots, d_n] \in \SP^n (M^g) \,:\,  [d_1, \ldots, d_n]  = [c_1, \ldots, c_n] + (f) \textrm{ for some } f\in \mathcal{M} (M^g)\}.$$
Here $(f)$ is the divisor consisting of the zeros and poles of $f$, and the addition is taken in the set of all divisors on $M^g$.
Letting $V_n\subset  \mathcal{M} (M^g)$ denote the subset of functions $f$ such that $ [c_1, \ldots, c_n] + (f)\in \SP^n (M^g)$, we see that $V_n$ is  the subspace of functions whose poles, with multiplicity, form a sub(multi)set of $[c_1, \ldots, c_n]$.  Moreover, the map $V_n\to \SP^n (M^g)$, $f\goesto  [c_1, \ldots, c_n] + (f)$ induces a homeomorphism $\mathbb{P}(V_n) \srm{\isom} \mu^{-1} (x)$, where $\mathbb{P}(V_n)$ denotes the complex projective space associated to $V_n$ (injectivity of this map comes from the fact that if two meromorphic functions have the same zeros and poles, their ratio is an entire function on the compact Riemann surface $M^g$, and must be constant).  The stabilization map on this fiber agrees with the map $\mathbb{P}(V_n) \to \mathbb{P}(V_{n+1})$ induced by the linear inclusion $V_n \subset V_{n+1}$.

The dimension of $V_n$ can be calculated using the Riemann--Roch theorem.
\end{proof}

\nd \emph{Proof of Theorem~\ref{stable-moduli-thm}.}
 The inclusion maps $\Un(n-1) \injects \Un(n)$ and $\SUn(n-1) \injects \SUn(n)$ all induce injective maps between the corresponding moduli spaces.  In fact, there exist triangulations of $\fX_{\Gamma^g} (\SUn(n))$ and $\fX_{\Gamma^g} (\Un(n))$ with
the images of $\fX_{\Gamma^g} (\SUn(n-1))$ and $\fX_{\Gamma^g} (\Un(n-1))$ as subcomplexes (see, for instance, the proof 
of~\cite[Lemma 5.7]{Ramras-stable-moduli}).
 It follows that there are isomorphisms
$$\pi_*  (\fX_{\Gamma^g} (\SUn)) \isom \colim_{n\to \infty} \pi_* ( \fX_{\Gamma^g} (\SUn(n))) \textrm{ and } 
\pi_*  (\fX_{\Gamma^g} (\Un)) \isom  \colim_{n\to \infty} \pi_* ( \fX_{\Gamma^g} (\Un(n))),$$
and moreover the map
\begin{equation}\label{SU} \pi_* ( \fX_{\Gamma^g} (\SUn) ) \maps \pi_* ( \fX_{\Gamma^g} (\Un) )\end{equation}
induced by the inclusions $\SUn(n)\injects \Un(n)$ is the colimit of the maps 
\begin{equation}\label{SU(n)}\pi_* ( \fX_{\Gamma^g} (\SUn(n))) \maps \pi_*  (\fX_{\Gamma^g} (\Un(n))).\end{equation}
By Corollary~\ref{univ-cover-surface}, we know that $\pi_0 ( \fX_{\Gamma^g} (\SUn(n))) = \pi_1  (\fX_{\Gamma^g} (\SUn(n))) = 0$, and
the maps (\ref{SU(n)}) are isomorphisms for $*\geqs 2$ by Corollary~\ref{pi}.  Passing to the colimits, we find that $\pi_0  (\fX_{\Gamma^g} (\SUn) ) = \pi_1  (\fX_{\Gamma^g} (\SUn) ) = 0$ and that the map (\ref{SU}) is an isomorphism for $*\geqs 2$.  It now follows from Theorem~\ref{U(n)-thm} that the homotopy groups of $\fX_{\Gamma^g} (\SUn)$ are trivial except in dimension 2, where $\pi_2 ( \fX_{\Gamma^g} (\SUn)) \isom \bbZ$.  The argument in~\cite[Lemma 5.7]{Ramras-stable-moduli} shows that $\fX_{\Gamma^g} (\SUn)$ has the homotopy type of a CW complex, so it is an Eilenberg--MacLane space of the form $K(\bbZ, 2)$, as is $\bbC \textrm{P}^\infty$.  Consequently, there exists a homotopy equivalence $\fX_{\Gamma^g} (\SUn)\heq \bbC \textrm{P}^\infty$.

To complete the proof, it suffices to show that the maps
$$\fX_{\bbZ^2} (\SUn(2))\srm{i} \fX_{\bbZ^2} (\SUn) \textrm{ and }  \fX_{\bbZ^2} (\SUn) \srm{p_g^*}  \fX_{\Gamma^g} (\SUn)$$
induce isomorphisms on $\pi_2$.    

First consider the map $i$.
There is a corresponding map $\fX_{\bbZ^2} (\Un(2)) \srm{j} \fX_{\bbZ^2} (\Un)$, and $i$ is the restriction of $j$ to the 
 fibers of the determinant maps $\fX_{\bbZ^2} (\Un(k)) \to \fX_{\bbZ^2} (\Un(1))$.  Under the homeomorphism $\fX_{\bbZ^2} (\Un (n))\srm{\isom} \SP^n (S^1\cross S^1)$ given by 
recording eigenvalues with respect to a simultaneous eigenbasis (see \cite[p. 1084]{Ramras-stable-moduli}, for instance),  the determinant map agrees with the map
$m \co \SP^n (S^1\cross S^1) \to S^1 \cross S^1$ induced by multiplication in the Abelian group $S^1\cross S^1$.  So we will just work with the map $m$.  Note that $m$ is a submersion (since for any Lie group $A$, the multiplication map $A\cross A\to A$ is a submersion) and hence a fiber bundle by Ehresmann's Fibration Theorem.  
The stabilization map $j$ corresponds to the map $\SP^n (S^1\cross S^1) \to \SP^{n+1} (S^1\cross S^1)$ sending
$(a_1, \ldots, a_n)$ to $(a_1, \ldots, a_n, e)$, where $e = (1,1)$ is the identity in $S^1\cross S^1$.  
(In the proof of \cite[Proposition 6.5]{Adem-Cohen-Gomez} it is stated that   $m$ can be identified with the Abel--Jacobi map $\mu$, but we do not need to use this.)
Since $\pi_i (S^1\cross S^1) = 0$ for $*>1$, the resulting diagram of long exact sequence in homotopy identifies the map
$\pi_2 ( \fX_{\bbZ^2} (\SUn(2))) \srm{i_*} \pi_2 (\fX_{\bbZ^2} (\SUn))$ with the map $\pi_2 ( \fX_{\bbZ^2} (\Un(2)) )\srm{j_*} \pi_2 (\fX_{\bbZ^2} (\Un))$.  Similarly, using the Abel--Jacobi maps $\mu_2$ and $\mu_n$ (which, by Lemma~\ref{AJ}, are fibrations with fibers $\bbC \textrm{P}^{1}$ and $\bbC \textrm{P}^{n-1}$, respectively) we obtain an identification of $j_*$ with the map
$\pi_2 (\bbC \textrm{P}^{1}) \to \pi_2 (\bbC \textrm{P}^{\infty})$ induced by the usual inclusion of $\bbC \textrm{P}^{1}\subset \bbC \textrm{P}^{\infty}$.  Since $\bbC \textrm{P}^1$ is the 3--skeleton of $\bbC \textrm{P}^{\infty}$, this map is an isomorphism.  Hence $i_*$ is an isomorphism, as desired.

To complete the proof, we will show that the pinch maps induce isomorphisms
 $\pi_2 (\fX_{\bbZ^2} (\SUn)) \xmaps{(p_g^*)_*}  \pi_2 (\fX_{\Gamma^g} (\SUn))$.  
 This will follow from the corresponding statement  with $\Un$ replaced by $\SUn$, since we have a commutative diagram
$$\xymatrix{ \pi_2 (\fX_{\bbZ^2} (\SUn)) \ar[r]^{(p_g^*)_*} \ar[d]^\isom & \pi_2 (\fX_{\Gamma^g} (\SUn))  \ar[d]^\isom\\
		 \pi_2 (\fX_{\bbZ^2} (\Un)) \ar[r]^{(p_g^*)_*}  & \pi_2 (\fX_{\Gamma^g} (\Un)),}
$$		 
where the isomorphisms come from (\ref{SU(n)}) (note that $\bbZ^2 = \Gamma^1$).
To analyze the map $\fX_{\bbZ^2} (\Un) \srm{p_g^*} \fX_{\Gamma^g} (\Un)$, we will use  facts from stable representation theory~\cite{Lawson-simul, Ramras-excision, Ramras-stable-moduli}.  These results are most easily stated for groups $\Gamma$ whose unitary representation spaces $\Hom(\Gamma, \Un(n))$ are path-connected; note that this condition holds for the trivial group $\{e\}$ and (as discussed in Section~\ref{surface-sec})   
for all  $\Gamma^g$ (including $\bbZ^2 = \Gamma^1$).  
(This condition immediately implies the ``stably group-like" condition discussed in \cite{Ramras-excision, Ramras-stable-moduli}.)
For such groups, the zeroth spaces of the spectra
 $\K (\Gamma)$ and $\Rdef (\Gamma)$ considered in the above references are naturally homotopy equivalent to 
 $\bbZ\cross \Hom(\Gamma, \Un)_{h\Un}$ and $\bbZ\cross \fX_{\Un} (\Gamma)$ (respectively), where 
 $$\Hom(\Gamma, \Un)_{h\Un} := \colim_{n\to\infty} E\Un(n)\cross_{\Un(n)} \Hom(\Gamma^g, \Un(n)),$$
is the colimit  of the \e{homotopy orbit spaces} for the conjugation action of $\Un(n)$ on $\Hom(\Gamma^g, \Un(n))$.
 These homotopy equivalences follow from~\cite[Corollary 4.4]{Ramras-excision} in the case of the spectrum $\K (\Gamma)$ (see also the proof of~\cite[Theorem 5.1]{Ramras-surface}), and from~\cite[Proposition 6.2]{Ramras-surface} in the case of the spectrum $\Rdef (\Gamma)$.  Since  $\K (\Gamma)$ and $\Rdef (\Gamma)$ are connective $\Omega$--spectra, their homotopy groups are naturally isomorphic to the homotopy groups of their zeroth spaces (and in particular, their negative homotopy groups are trivial).  Note that $\pi_0 (\bbZ \cross \Hom(\Gamma, \Un)_{h\Un})$ has a natural addition given by the usual addition in $\bbZ$ and the block sum operation on unitary matrices, and the isomorphism respects this structure \cite[Theorem 3.6]{Ramras-excision}.

 For each finitely generated group $\Gamma$, Lawson constructed a natural long exact sequence ending with
$$\pi_{0} \K (\Gamma) \srm{\beta_*} \pi_{2} \K (\Gamma) \srm{\pi_*} \pi_2 \Rdef (\Gamma) \maps \pi_{-1} \K (\Gamma) = 0.$$
This is the long exact sequence in homotopy resulting from Lawson~\cite[Corollary 4]{Lawson-simul}; note that Lawson uses the notation $R[\Gamma]$ in place of $\Rdef (\Gamma)$.  (Here $\beta_*$ is the Bott map, raising degree by 2, and $\pi_*$ is induced by the quotient maps $\Hom(\Gamma, \Un (n))\to \fX_\Gamma (\Un(n))$.)
The pinch map induces a map between the Lawson exact sequences for $\bbZ^2$ and $\Gamma^g$, and by the 5--lemma it will suffice to show that the induced maps $\pi_{i} \K (\bbZ^2) \xmaps{(p_g^*)_*} \pi_{i} \K (\Gamma^g)$ are isomorphisms for $i = 0, 2$.  

First we consider $\pi_0$. Both $\pi_{0} \K (\Gamma^g)$ and $\pi_{0} \K (\bbZ^2)$  are isomorphic to $\bbZ$, because the zeroth spaces of these spectra have the form $\bbZ \cross \Hom(\Gamma, \Un)_{h\Un}$, and we have observed already that $\Hom(\Gamma, \Un)$ is connected for these groups.  The same reasoning shows that $\pi_{0} \K (\{e\}) \isom \bbZ$.
Now consider the maps 
$\{e\} \to \Gamma^g \to \{e\}$.  These induce maps on $\pi_0 \K (-)$ whose composite is the identity.  Since all the groups in question are infinite cyclic, these maps must be isomorphisms.  From the diagram
$$\xymatrix{& \pi_0 \K (\{e\}) \ar[dr]^\isom \ar[dl]^\isom \\
\pi_0 \K  (\bbZ^2) \ar[rr]^{(p_g^*)_*}  &   & \pi_{0} \K (\Gamma^g) }
$$
we now conclude that $p_g^*$ induces an isomorphism on $\pi_0 \K (-)$.

Now consider the map $\pi_{2} \K (\bbZ^2) \xmaps{(p_g^*)_*} \pi_{2} \K ( \Gamma^g)$.  As discussed above, this map is equivalent to the map $\pi_{2} (\Hom(\bbZ^2, \Un)_{h\Un})\xmaps{(p_g^*)_*} \pi_{2} (\Hom(\Gamma^g, \Un)_{h\Un})$.  
For any $G$--space $X$, the projection map $X_{hG} \to BG = EG/G$ is a fibration with fiber $X$.
 If there exists a point $x\in X$ that is fixed by all $g\in G$, then the map $BG = EG/G\to X_{hG}$, $[b] \goesto [b, x]$, defines a continuous section of this fibration, and hence the boundary maps in the long exact sequence on homotopy are all zero.  This applies to the cases in question, since the trivial representation is fixed by conjugation.
A $G$--equivariant map $X\to Y$ induces a map $X_{hG} \to Y_{hG}$ covering the identity on $BG$. 
Thus $p_g$ induces a diagram of short exact sequences in homotopy 
$$\xymatrix{ 0\ar[r]   & \pi_2 (\Hom(\bbZ^2, \Un)) \ar[r] \ar[d]^{(p_g^*)_*} &  \pi_2 (\Hom(\bbZ^2, \Un)_{h\Un}) \ar[r] \ar[d]^{(p_g^*)_*}
						& \pi_2 (B\Un) \ar[d]^= \ar[r] & 0 \\
	0\ar[r]  &  \pi_2 (\Hom(\Gamma^g, \Un)) \ar[r] &  \pi_2 (\Hom(\Gamma^g,\Un)_{h\Un}) \ar[r] 
						& \pi_2 (B\Un)  \ar[r] & 0,}
$$
so it will suffice to show
that $\pi_{2} ( \Hom(\bbZ^2, \Un))\xmaps{(p_g^*)_*} \pi_{2} (\Hom(\Gamma^g, \Un))$ is an isomorphism.

We will compare this map with (reduced) topological $K$--theory and integral cohomology. 
Let $\Map_*^0 (X, Y)$ denote the space of nullhomotopic based maps from $X$ to $Y$.
 For any finitely generated group $\Gamma$, there is a natural map 
\begin{equation}\label{B}B \co \Hom(\Gamma, \Un) \maps \Map^0_* (B\Gamma, B\Un),\end{equation}
where $B\Gamma$ and $B\Un$ are the (functorial) simplicial models for the classifying spaces of these groups.
 When $\Gamma = \Gamma^g$ ($g\geqs 1$), the map (\ref{B}) induces an isomorphism on homotopy groups in positive dimensions \cite[Theorem 3.4]{Ramras-stable-moduli}.  The homomorphism $(p_g)_* \co \Gamma^g \to \bbZ^2$ induces a map 
$B(p_g)_* \co B\Gamma^g \to B\bbZ^2$, 
so it suffices to check that the induced map
$$\pi_{2} ( \Map_* (B\bbZ^2, B\Un)) \maps \pi_{2} (\Map_* (B\Gamma^g, B\Un))$$
is an isomorphism.
There are natural isomorphisms $\pi_* (\Map_* (X, B\Un)) \isom \wt{K}^{-*} (X)$ for each CW complex $X$ and each $*>0$, so 
it will suffice to show that $B(p_g)_*$ induces an isomorphism on $ \wt{K}^{0} (-) \isom \wt{K}^{-2} (-)$.

The classifying space $B\Gamma$ is an Eilenberg--MacLane space of the form $K(\Gamma, 1)$, and there is a natural isomorphism $\Gamma\srm{\isom} \pi_1 (B\Gamma)$ such that for any group homomorphism $f\co \Gamma\to \Gamma'$, the map $(Bf)_* \co \pi_1 (B\Gamma) \to \pi_1 (B\Gamma')$ agrees with $f$ under these isomorphisms.  This gives us a commutative diagram of groups
$$\xymatrix{ \Gamma^g \ar[rr]^{(p_g)_*} \ar[d]^\isom &\, & \bbZ^2 \ar[d]^\isom\\
		\pi_1 (B\Gamma^g) \ar[rr]^{(B(p_g)_*)_*} &\, & \pi_1 B\bbZ^2,}
$$
and since based homotopy classes of based maps between Eilenberg--MacLane spaces are in 1-1 correspondence with homomorphisms between the underlying homotopy groups \cite[Section 4.3, Exercise 4]{Hatcher}, there exists a corresponding diagram
$$\xymatrix{ M^g \ar[rr]^{p_g} \ar[d]^\isom &\, & M^1 \ar[d]^\isom\\
		 B\Gamma^g  \ar[rr]^{ B(p_g)_*} &\, & B\bbZ^2,}
$$
commuting up to homotopy (note that $M^g$ is an Eilenberg--MacLane space for $g\geqs 1$).  Thus it suffices to show that $\wt{K}^{0} (M^1)\xmaps{p_g^*} \wt{K}^{0} (M^g)$ is an isomorphism.  A complex vector bundle $E$ on a surface is determined up to isomorphism by its first Chern class $c_1 (E)$, and since $c_1 (E\oplus F) = c_1 (E) + c_1 (F)$   we obtain a commutative diagram 
$$\xymatrix{ K^0 (M^1) \ar[r]^{p_g^*} \ar[d]^\isom_{c_1} & K^0 (M^g) \ar[d]^\isom_{c_1} \\
		 H^2 (M^1; \bbZ) \ar[r]^{p_g^*}   & H^2 (M^g ; \bbZ).}
$$
Since $p_g$ can be realized as the map $M^g = (S^1\cross S^1) \# M^{g-1} \to (S^1\cross S^1)$ that collapses $M^{g-1}$ to a point, it sends the fundamental class of $M^g$ to that of $S^1\cross S^1$.  So the induced map $p_g^*$ is an isomorphism on $H^2(-; \bbZ)$, and we are done.
$\hfill \Box$

\begin{remark} The same argument can be used to show that if $f\co M^g\to M^h$ is a degree $d$ map, then the induced map
$\fX_{\Gamma^h} (\SUn) \srm{f^*} \fX_{\Gamma^g} (\SUn)$ has degree $\pm d$ when thought of as a map between Eilenberg--MacLane spaces of the form $K(\bbZ, 2)$.
\end{remark}

It would be interesting to have an explicit  homotopy equivalence $\fX_{\Gamma^g} (\SUn) \to \bbC \textrm{P}^\infty$.  Such a map would correspond to a class in $H^2 (\fX_{\Gamma^g} (\SUn); \bbZ)$, or equivalently to a complex line bundle on $\fX_{\Gamma^g} (\SUn)$.   Goldman~\cite{Goldman-symplectic} constructed a symplectic form $\omega$ on the (Zariski) tangent spaces to the varieties $\fX_{\Gamma^g} (\SUn(n))$.  Restricted to the non-singular locus (the subspace of irreducible representations) this form represents a 2-dimensional de Rham cohomology class $[\omega]$.  Ramadas, Singer, and Weitsman~\cite{RSW} constructed a complex line bundle over $\fX_{\Gamma^g} (\SUn)$ whose first Chern class is a non-zero scalar multiple of $[\omega]$.  These bundles are compatible with stabilization, and hence yield a line bundle $L\to \fX_{\Gamma^g} (\SUn)$.

\begin{question}$\label{RSW}$
Is the classifying map $\fX_{\Gamma^g} (\SUn)\to \bbC \textrm{P}^\infty$ for the Ramadas--Singer--Weitsman line bundle a homotopy equivalence?
\end{question}

We note that since $\fX_{\Gamma^g} (\SUn)\heq \bbC \textrm{P}^\infty$,  it would suffice to check whether this classifying map induces an isomorphism on $H^2 (-; \bbZ)$.  In fact, since the map
$$S^2 \isom \fX_{\bbZ^2} (\SUn(2)) \srm{f} \fX_{\Gamma^g} (\SUn)$$
described in Theorem~\ref{stable-moduli-thm} induces an isomorphism on $\pi_2$, it also induces an isomorphism on $H^2 (-; \bbZ)$.  Thus to answer Question~\ref{RSW}, it would suffice to describe the classifying map for the pullback $f^*(L) \to S^2$; in other words, it would suffice to calculate the first Chern class of this line bundle on $S^2$.  Since $L$ is defined using the holonomy fibration $\mathcal{A}_{\mathrm{flat}}(M^g\cross \SUn (n)) \to \fX_{\Gamma^g} (\SUn (n))$
and the Chern--Simons functional, this calculation appears to be an interesting problem (here 
 $\mathcal{A}_{\mathrm{flat}} (M^g\cross \SUn (n))$ denotes the space of flat connections on the trivial $\SUn (n)$--bundle over $M^g$).\footnote{After we posed this question, it was addressed in \cite{JeRaWe}.}

\section{Appendix by Nan-Kuo Ho \& Chiu-Chu Melissa Liu: \texorpdfstring{\\}{}
Connected components of surface  group representations for complex 
reductive Lie groups }\label{reductivecomponents}

Let $G$ be a connected reductive $\bbC$--group. Then $G$ is the complexification of a compact connected Lie group $K$, i.e., $G=K^\bC$.

Let $H=Z(K)_0$ be the identity component of the center $Z(K)$ of $K$, and let $DK=[K,K]$ be the commutator subgroup of $K$.
Then $H\cong U(1)^m$ is a compact torus of dimension $m\in \bZ_{\geq 0}$,  and $DK$ is the maximal connected
semisimple subgroup of $K$. The map
\[
\phi: H\times DK \rahl K = H \cdot DK,\quad (h,k)\mapsto hk
\]
is a finite cover which is also a group homomorphism; the kernel of $\phi$ is isomorphic to $H\cap DK$, which is a finite Abelian group.
Let $\rho_{ss}: \wt{DK} \rahl DK$ be the universal covering map, which is also a group
homomorphism. Then $\wt{DK}$ is a compact, connected, simply connected, semisimple Lie group, and $\Ker(\rho_{ss})\cong \pi_1(K)$ is
a finite Abelian group.   Let $\fh\cong \bR^m$ be the Lie algebra of $H$ and let $\exp_H:\fh\rahl H$ be
the exponential map.  We have a universal covering map
\[
\rho:~ \tK:=\fh\times\wt{DK}\rahl K=H \cdot DK,\quad (X,k)\mapsto \exp_H(X)\rho_{ss}(k).
\]

Let $H^\bC$, $\fh^\bC$, $DG$, and $\wt{DG}$ be the complexification of $H$, $\fh$, $DK$, and $\wt{DK}$, respectively.
Then $H^\bC\cong (\bC^*)^m$, $\fh^\bC\cong \bC^m$, $DG$ is
a semisimple complex Lie group, and $\wt{DG}$ is a simply connected
semisimple complex Lie group. We have a commutative diagram
$$
\begin{CD}
1 @>>> \Ker(\rho_{ss})  @>>> \wt{DK} @>{\rho_{ss}}>> DK @>>> 1 \\
  && @VV{\cong}V  @V{\tilde{j}}VV  @V{j}VV \\
1 @>>> \Ker(\rho_{ss}^\bC) @>>> \wt{DG} @>{\rho_{ss}^\bC}>> DG @>>> 1
\end{CD}
$$
where
\begin{itemize}
\item the rows are short exact sequence of groups;
\item $\rho_{ss}$ and $\rho_{ss}^\bC$ are universal covering maps;
\item $j$ and $\tilde{j}$ are inclusion maps and homotopy
equivalences;
\item $\Ker(\rho_{ss}^\bC) =\Ker(\rho_{ss})\cong \pi_1(DK)\cong\pi_1(DG)$
\end{itemize}
We have a universal covering map
\begin{equation}\label{eqn:rhoC}
\rho^\bC:~ \tG:=\fh^\bC\times\wt{DG}\rahl G=H^\bC \cdot DG,\quad (X,g)\mapsto
\exp_{H^\bC}(X)\rho^\bC_{ss}(g),
\end{equation}
where $\exp_{H^\bC}: \fh^{\bC}\to H^\bC$ is the exponential map.

Let $\Si$ be a Riemann surface of genus $\ell\geq 2$.
Jun Li showed in \cite{Li-surface-groups} that for a complex connected semissimple Lie group $DG$,
\[
\pi_0(\Hom(\pi_1(\Si),DG))\cong \pi_1(DG).
\]
It is known that for a compact connected Lie group $K$ (cf: \cite{A-B},\cite{Ho-Liu-ctd-comp-II}),
\[
\pi_0(\holK)\cong\pi_1(DK).
\]
In this appendix, we prove the following statement.
\begin{proposition}
Let $G$ be a connected reductive $\bbC$--group, and let $DG$ be its maximal connected
semisimple subgroup. Then
\[
\pi_0(\holG)\cong \pi_1(DG).
\]
\end{proposition}
\begin{proof}
Define a map $O_2:\holG\rahl \tG$ by $(\ab)\in G^{2\ell}\mapsto \tabab$
where $\ta_i$, $\tb_i$ are any liftings of  $a_i$, $b_i$ under
the universal covering $\rho^\bC:\tG\to G$.
Notice that $\tabab$ is independent of the choice of the liftings
$\ta_i$, $\tb_i$, so this map is well-defined. Moreover,
$\tabab \in \Ker(\rho^\bC) \cap (\{0\}\times \wt{DG}) = \{0\}\times \Ker(\rho_{ss}^\bC)
\cong \pi_1(DG)$. We obtain a map
$$
O_2: \holG\to  \{0\}\times \Ker(\rho_{ss}^\bC)\cong \pi_1(DG)
$$
which restricts to a map
$$
O'_2:\holGss \to \{0\}\times \Ker(\rho_{ss}^\bC)\cong \pi_1(DG)
$$
By the results in \cite{Li-surface-groups}, for each $(0,k)\in\{0\}\times \Ker(\rho_{ss}^\bC)$,
$(O'_2)^{-1}(0,k)$ is nonempty and connected. Since
$(O'_2)^{-1}(0,k)\subset O_2^{-1}(0,k)$, $O_2^{-1}(0,k)$ is nonempty
for each $k\in \Ker(\rho^\bC_{ss})$. It remains to show that
$$
O_2^{-1}(0,k)=\{(\ab)\in G^{2\ell},\tabab=(0,k)\}
$$
is connected for each $k\in \Ker(\rho_{ss}^\bC)$. Since $(O_2')^{-1}(0,k)$
is connected, it suffices to prove that, for any $(\ab)\in O_2^{-1}(0,k)$, there exists a path $\gamma:[0,1]\to O_2^{-1}(0,k)$ such that
$\gamma(0)\in (O_2')^{-1}(0,k)$ and $\gamma(1)=(\ab)$.

Given $(\ab)\in O_2^{-1}(0,k)$,
write $a_i=\exp(\xi_i) a_i'$, $b_i=\exp(\eta_i) b_i'$, where $\xi_i,~\eta_i\in\fh^\bC$
and $a'_i,~b'_i \in DG$. Define a path $\gamma: [0,1]\to G^{2\ell}$ by
$\gamma(t)=(a_1(t),b_1(t),\ldots, a_\ell(t), b_\ell(t))$, where
$a_i(t)=\exp(t\xi_i)a'_i$ and $b_i(t)=\exp(t\eta_i) b'_i$.
Then $\gamma(t)\in O_2^{-1}(0,k)$ since $\prod_{i=1}^{\ell}[\ta_i(t),\tb_i(t)]=\prod_{i=1}^\ell[\ta_i,\tb_i]=(0,k)$. Moreover,
$a_i(0)=a_i'$, $b_i(0)=b_i'$ and $a_i(1)=a_i$, $b_i(1)=b_i$, so $\gamma(0)\in (O_2')^{-1}(0,k)$ and $\gamma(1)=(\ab)$. 
Thus, any point in $O_2^{-1}(0,k)$ is path connected to the set $(O_2')^{-1}(0,k)$. This completes the proof of the proposition.
\end{proof}

\begin{remark}
In the above proof, the choices of $(\xi_i,a_i'),~(\eta_i,b_i')\in \fh^\bC\times DG$ are not unique, 
but it only means that we have a different path $\gamma(t)$ connecting the point $(\ab)$ to some 
{\rm(}possibly different{\rm)} point $(a_1',b_1',\ldots, a_\ell',b_\ell')$ in the set $(O'_2)^{-1}(0,k)$.
\end{remark}

\end{document}